\documentclass[hidelinks,onefignum,onetabnum]{siamart251104}

\usepackage{lipsum}
\usepackage{amsfonts}
\usepackage{graphicx}
\usepackage{epstopdf}
\usepackage{algorithm,algorithmic}
\usepackage{verbatim}
\usepackage{multirow}
\usepackage{subcaption}
\usepackage{enumitem}
\usepackage{amsmath,booktabs,ctable,threeparttable}
\usepackage{amssymb,boxedminipage}
\usepackage{microtype}
\allowdisplaybreaks[4]
\ifpdf
  \DeclareGraphicsExtensions{.eps,.pdf,.png,.jpg}
\else
  \DeclareGraphicsExtensions{.eps}
\fi

\newsiamremark{remark}{Remark}
\newsiamremark{hypothesis}{Hypothesis}
\crefname{hypothesis}{Hypothesis}{Hypotheses}
\newsiamthm{claim}{Claim}
\newsiamremark{fact}{Fact}
\crefname{fact}{Fact}{Facts}

\crefname{hypothesis}{Hypothesis}{Hypotheses}

\usepackage{amsopn}
\DeclareMathOperator{\Diag}{diag}
\DeclareMathOperator{\Rank}{rank}

\DeclareMathOperator{\Range}{\mathcal{R}}

\newcommand{\Complex}{\mathbb{C}}
\newcommand{\Real}{\mathbb{R}}

\newcommand{\norm}[2][]{\Vert {#2} \Vert_{#1}}
\newcommand{\normbig}[2][]{\left\Vert {#2} \right\Vert_{#1}}
\newcommand{\abs}[1]{\vert {#1} \vert}

\headers{Restarting the SS--RR and SS--RRR methods}{Z. Jia and T. Liu}

\title{Stable and effective restarting of the SS--RR and SS--RRR methods
for computing partial eigenpairs of Hermitian matrices
\thanks{Submitted to the editors DATE.
\funding{The work was supported by the National Natural Science Foundation of China (NSFC) under
grants 12571404 and 12671453.}}}

\author{Zhongxiao Jia\thanks{Corresponding author. Department of Mathematical Sciences, Tsinghua University, 100084 Beijing, China
  (\email{jiazx@tsinghua.edu.cn}).}
\and Tianhang Liu\thanks{Department of Mathematical Sciences, Tsinghua University, 100084 Beijing, China
  (\email{lth@tsinghua.edu.cn}).}
}

\begin{document}
\maketitle

\begin{abstract}
We consider stable and effective restarts of the SS--RR and SS--RRR methods for
computing the eigenpairs of a large Hermitian matrix $A\in\mathbb{C}^{n\times n}$ with the
$n_{ev}$ eigenvalues in a given interval. Taking
the CJ--SS--RR and CJ--SS--RRR methods for instance, we prove that
the commonly used restart scheme and its orthonormalized variant, named the
natural restart schemes, are not numerically stable,
thereby lowering attainable accuracy of approximate
eigenpairs and missing some of the $n_{ev}$ desired eigenpairs.
To cure these deficiencies,
we propose a new restart scheme, called the combination restart, that
takes the $n\times \ell$ starting matrix to be the product of the matrix, which consists of
the selected $n_{ev}$ approximate eigenvectors, and
an $n_{ev}\times \ell$ Gaussian random
matrix. We prove the remarkable mathematical and numerical
superiority of the combination restart over the natural restarts,
then present a hybrid restart scheme that fully exploits the advantages of the natural restart
and combination restart, and finally develop CJ--SS--RR and CJ--SS--RRR algorithms with the hybrid restart.
Numerical experiments confirm that the new algorithms are numerically backward stable to
the working precision and more efficient than the CJ--SS--RRR algorithm with the natural restart. The combination restart and hybrid restart schemes can be straightforwardly extended to the
SS--RR and SS--RRR methods for general standard and generalized eigenvalue problems.

\end{abstract}



\section{Introduction}

For a large Hermitian matrix $A \in \Complex^{n \times n}$, consider
the eigenvalue problem $ A x = \lambda x$ with the 2-norm $\|x\|=1$.
We are interested in all the $n_{ev}$ eigenvalues of $A$ in a given interval $[a, b]$ counting multiplicities and their corresponding eigenvectors,
where the interval $[a, b]$ is contained in the spectrum interval of $A$.
Such problems arise in many disciplines, e.g.,
linear response eigenvalue problems in quantum chemistry \cite{LSEP},
and extraction of electromagnetic modes in target frequency bands \cite{Jin_FEM_EM},
and electronic structure calculations \cite{FEAST2009}, to mention only a few.

Over the past two decades, some
eigensolvers based on contour integrals,
such as the SS-type methods~\cite{block_SS_RR,SS_theory2010,SS_Arnoldi,SS_Hankel,SS_RR}
and the FEAST eigensolver \cite{FEAST2009},
have been proposed and developed for computing
the eigenpairs of a large matrix or matrix pair with the eigenvalues in a given region.
The earliest SS-type method is the SS--Hankel method \cite{SS_Hankel},
which was proposed by Sakurai and Sugiura
for solving the generalized eigenvalue problem of non-Hermitian matrix pairs.
By employing certain contour integrals, which correspond to the zeroth to higher-order moments,
the SS--Hankel method reduces the large generalized eigenvalue problem to
a much smaller one of a Hankel matrix pair,
from which the approximate eigenpairs of the original matrix pair are computed.
However,
because of the high ill-conditioning of the Hankel matrices~\cite{beckermann2000},
the SS--Hankel method suffers from numerical instability~\cite{SS_RR}.
Subsequently,
the SS--RR method \cite{SS_RR} and the SS--Arnoldi method \cite{SS_Arnoldi} were proposed
that are based on contour integrals to construct a search subspace,
apply the Rayleigh--Ritz projection to it, and compute the approximate eigenpairs.
These methods have their corresponding block versions \cite{block_SS_RR,SS_theory2010,SS_Arnoldi}.

The FEAST eigensolver proposed by Polizzi~\cite{FEAST2009}
for solving the Hermitian positive definite generalized eigenvalue problems is based on the
contour integral corresponding to the zeroth moment, which is closely related to
the spectral projector associated with the eigenvalues in the given interval.
For the Hermitian positive definite matrix pair,
the FEAST eigensolver has a variant based on the Zolotarev rational functions \cite{FEAST_Zolotarev}.
Later on, the FEAST eigensolver~\cite{FEAST_non_Hermitian,FEAST_oblique} was extended to
the generalized eigenvalue problem of a non-Hermitian matrix pair.

In the above contour integral-based methods,
the contour integrals are approximated by some numerical quadratures or the Zolotarev rational functions,
which require solving a large number of shifted complex linear systems at each iteration
and dominate the computational cost of FEAST solvers. Gavin and Polizzi \cite{IFEAST} apply the inexact solves for the shifted linear systems
involved in the FEAST solver
and propose an IFEAST solver,
which allows the shifted linear systems to be solved inexactly at each iteration
and gradually increases the accuracy requirement as the iteration proceeds,
thereby reducing the computational cost of iteratively solving the linear systems.
Since the region of interest is inside the matrix spectrum,
these linear systems may be highly indefinite,
and, in the meantime, if some shift is close to an eigenvalue
then the corresponding coefficient matrix must be very ill-conditioned. As a result,
the convergence of Krylov iterative methods,
such as GMRES and BiCGStab~\cite{Golub_Matrix}, is generally very slow and time consuming, even though
the solution accuracy is not high.
Moreover, in finite precision arithmetic,
the conditioning of these shifted linear systems critically affects the convergence of
the methods and attainable accuracy, as has been analyzed
and numerically confirmed in
\cite{CJ_FEAST_cross,CJ_FEAST_augmented,CJ_FEAST_GSVD}.

To overcome the aforementioned deficiencies,
Jia and Zhang~\cite{CJ_FEAST_cross,CJ_FEAST_augmented,CJ_FEAST_GSVD} make
use of the Chebyshev--Jackson (CJ) series expansion to approximate the step function associated with the
zeroth contour integral,
and establish compact pointwise convergence of the CJ series expansion.
They have developed the CJ-FEAST solvers for computing a partial singular value decomposition (SVD) and generalized singular value decomposition (GSVD) with the (generalized) singular values in a given
interval. A key difference between them and the contour integral-based FEAST solvers
is that proper CJ series expansions are used to construct approximations to the
underlying spectral projectors without solving any shifted linear system. Furthermore,
the CJ-FEAST solvers can compute the desired SVD and GSVD more accurately,
and can be directly adapted to the eigenvalue problems of a Hermitian matrix
and Hermitian positive definite matrix pair.
The authors of this paper extend the CJ series expansion from the step function to the piecewise continuous
functions that correspond to the zeroth to the higher-order moments,
and propose the CJ series-based block SS--RR (CJ--SS--RR) method~\cite{CJ_SS_RR}, which has numerically
been shown to be considerably more efficient and robust than the original SS--RR algorithm.
Subsequently, to resolve the issue of spurious Ritz values and improve the reliability and overall
efficiency of the CJ--SS--RR method, based on the refined Rayleigh--Ritz projection, the authors propose
a refined CJ--SS--RR method, abbreviated as CJ--SS--RRR,
and present a reliable removal approach for the spurious Ritz values \cite{CJ_SS_RRR}. Numerical experiments
illustrate that the CJ--SS--RRR algorithm is more efficient and reliable than the CJ--SS--RR
algorithm \cite{CJ_SS_RRR}.

The FEAST solvers fall into subspace iteration applied to the approximate spectral projector,
and they restart in the natural way that premultiplies
the current subspace by the approximate spectral projector \cite{FEAST_SI}.
By contrast, the SS-type methods were originally proposed as one-shot methods,
and no restart is used in their original forms.
To make relative residual norms of approximate eigenpairs
drop below some prescribed tolerance, the dimension of the search subspace may be
large, or one needs to use a sufficiently accurate numerical quadrature with
a lot of nodes. In the latter case, one must solve a lot of shifted linear systems accurately,
which may be too expensive.

A common compromise is to restrict the dimension of the search subspace and use a less accurate
numerical quadrature with not many nodes. Consequently, the subspace produced in such a way
generally does not contain sufficiently accurate information on the desired eigenvectors, so that
restart is necessary.
Sakurai, Futamura, and Tadano \cite{SS_restart1} propose a restart scheme for the SS-type methods,
which is shown to be equivalent to the one used in subspace iteration \cite{SS_theory2016,SS_review}. This restart scheme, called the natural restart,
is also adopted in the CJ--SS--RR method \cite{CJ_SS_RR} and the CJ--SS--RRR method \cite{CJ_SS_RRR}.
In exact arithmetic, such a restarted block SS--RR algorithm is a specific subspace iteration, and generates
a sequence of subspaces that contain increasingly more accurate approximate eigenvectors.
Unfortunately, in finite precision arithmetic,
the restarted block SS--RR method generally performs poorly
because it ignores the block Krylov-like structure of updated search subspaces
whose basis becomes increasingly ill-conditioned as restarts proceed.
We will rigorously prove this property in this paper. As a result, such kinds of restarted SS-type methods,
including the CJ--SS--RR and CJ--SS--RRR methods,
suffer from possibly severe numerical instability,
which ultimately causes the computed subspaces to deviate from the true ones or the dimension
of effective subspaces to be smaller than $n_{ev}$,
thereby lowering attainable accuracy of approximate eigenpairs and missing
some of the desired eigenpairs. That is,
the algorithms are not numerically backward stable to the working precision
and may not find all the $n_{ev}$ desired eigenpairs.


To avoid numerical instability and improve the computational accuracy, in this paper,
we adapt a restart scheme in \cite{jia_block} to our context, and propose a
combination restart that takes the $n\times \ell$ starting matrix to be the product
of the $n\times n_{ev} $ matrix, which consists of
properly selected $n_{ev}$ approximate eigenvectors, and an $n_{ev}\times \ell$
Gaussian random matrix, where $\ell<n_{ev}$ is the block size of SS-type methods,
so that the search subspace retains the information on the $n_{ev}$ desired eigenvectors
as much as possible. We prove that the new updated search subspaces retain the
effective information on all the $n_{ev}$ desired eigenvectors and the resulting
restarted algorithms are always numerically stable to the working precision.
Furthermore, we propose a hybrid restart scheme that combines the advantages of the natural
and combination restarts, and develop numerically stable and efficient restarted
CJ--SS--RR and CJ--SS--RRR algorithms
that not only converge faster than the restarted algorithms with either the natural
or combination restart but also are numerically backward stable to the working precision.

The paper is organized as follows.
In \cref{sec: previous works},
we review the CJ series expansion and some of its properties, establish a new property,
and briefly describe the CJ--SS--RRR method.
\Cref{sec: restart strategies} analyzes the deficiency of the natural restarts
and proposes the combination restart.
In \cref{sec: CJ--SS--RRR with hybrid restart},
we present the hybrid restart scheme and develop CJ--SS--RR and CJ--SS--RRR algorithms with the hybrid restart.
Numerical experiments are reported in \cref{sec: numerical experiments} to illustrate the performance of
the algorithms with different restarts and justify the superiority
of the CJ--SS--RRR algorithm with the hybrid restart.
We conclude the paper in \cref{sec: conclusion}.

Throughout the paper,
we denote by $\norm{\cdot}$ the 2-norm of a vector or matrix,
by $\Range(X)$ the column space of a matrix $X$, by the superscripts $T$ and $H$
the transpose and conjugate transpose of a matrix or
vector, by $\lambda(A)$ the set of eigenvalues of $A$,
and by $\lambda_{\min}$ and $\lambda_{\max}$ the minimum and maximum eigenvalues of $A$, respectively.
We also denote by $\sigma_i(X)$ the $i$-th largest singular value of a matrix $X$,
by $\sigma_{\max}(X)$ and $\sigma_{\min}(X)$ the maximum and minimum singular values,
and by $\kappa(X)=\sigma_{\max}(X)/\sigma_{\min}(X)$ the spectral condition number of $X$.
Since one can use a linear transformation
\begin{displaymath}
  l(t) = \frac{2t - \lambda_{\max} - \lambda_{\min}}{\lambda_{\max} - \lambda_{\min}}
\end{displaymath}
to map $[\lambda_{\min}, \lambda_{\max}]$ to $[-1, 1]$, we assume, for brevity
and without loss of generality,
that $\lambda_{\min} = -1$ and $\lambda_{\max} = 1$ when presenting most of the results and analysis.

\section{The CJ series expansion and the CJ--SS--RR and CJ--SS--RRR methods}
\label{sec: previous works}

We first review the CJ series expansion and some available results to be used
in this paper, and
establish a new property. We then describe the CJ--SS--RR and CJ--SS--RRR methods and
the removal approach for spurious (refined) Ritz values and extra (refined) Ritz vectors.

\subsection{The CJ series expansion}

Given an interval $[a, b] \subset [-1, 1]$, define the step function
\begin{equation} \label{eq: def: step-function}
  h(t) = \begin{cases}
    1,           & t \in (a, b) , \\
    \frac{1}{2}, & t \in \{ a, b \}, \\
    0,           & t \in [-1, 1] \setminus [a, b],
  \end{cases}
\end{equation}
whose values at the ends $a$ and $b$ are the means of the left and right limits of $h(t)$ at the two ends.
Let $p$ be a real polynomial.
Then the $d$-degree CJ series expansion of the piecewise continuous function $p(t)h(t)$ defined in \cite{CJ_SS_RR} is given by
\begin{equation} \label{eq: def: F_d}
  F_d(p)(t) := \frac{c_0}{2} +  \sum_{j=1}^d \rho_{j, d} c_j T_j(t) \approx p(t)h(t),
\end{equation}
where $T_j$ is the $j$-degree Chebyshev polynomial of the first kind \cite{ApproxIntroduction}, $c_j$ is the Chebyshev
coefficient given by
\begin{equation} \label{eq: def: Chebyshev coefficient}
  c_j = \frac{2}{\pi} \int_{-1}^1 \frac{p(t)h(t) T_j(t)}{\sqrt{1-t^2}}  {\rm d} t = \frac{2}{\pi} \int_a^b \frac{p(t) T_j(t)}{\sqrt{1-t^2}} {\rm d} t,
\end{equation}
and $\rho_{j,d}$ is the Jackson damping factor (cf. \cite{eig_count_Saad,Jackson_damp}) given by
\begin{equation} \label{eq: def: Jackson damping factor}
  \rho_{j, d} = \frac{\sin((j+1)\alpha_d)}{(d+2)\sin(\alpha_d)} + \left( 1-\frac{j+1}{d+2} \right) \cos(j\alpha_d), \quad \text{where } \alpha_d = \frac{\pi}{d+2}.
\end{equation}


We unify Theorem 4.1 of \cite{CJ_FEAST_cross} and Theorem 3.3 of \cite{CJ_SS_RR} as follows,
which describe the convergence of $F_d(p)(t)$ to $p(t)h(t)$ for $p\equiv 1$ and a general
real polynomial $p$, respectively,
and the eigenvalue distribution of $F_d(1)(A)$.

\begin{theorem} \label{thm: eigenvalues of F_d(1)(A)}
  Let
  \begin{equation} \label{eq: def: Delta_min and delta}
    \Delta_{\min}
    = \min_{ \lambda \in \lambda(A) \setminus \{a, b\} }
      \{ |\arccos(\lambda) - \alpha|,
         |\arccos(\lambda) - \beta| \}, \quad
    \delta = \frac{\pi^6}{2(d+2)^3\Delta_{\min}^4},
  \end{equation}
  where $\alpha = \arccos(a)$ and $\beta = \arccos(b)$.
  Then, for each of the cases (i) $p(s) \equiv 1, \lambda \in \lambda(A)$ and (ii) $\max_{s \in [a, b]} |p(s)| \leq 1, \lambda \in \lambda(A) \setminus [a, b]$, it holds that
  \begin{equation} \label{eq: error bound of F_d(p)(t)}
    \abs{F_d(p)(\lambda) - p(\lambda)h(\lambda)} \leq \delta.
  \end{equation}
  In particular,
  if the degree $d$ of the CJ series expansion satisfies
  \begin{equation} \label{eq: theoretical requirement on degree}
    d > \frac{\sqrt[3]{2}\pi^2}{\Delta_{\min}^{4/3}} - 2,
  \end{equation}
  then for all eigenvalues $\lambda_{in} \in \lambda(A) \cap [a, b]$ and $\lambda_{out} \in \lambda(A) \setminus [a, b]$ it holds that
  \begin{displaymath}
    1 \geq F_d(1)(\lambda_{in}) > \frac{1}{4} > F_d(1)(\lambda_{out}) \geq 0.
  \end{displaymath}
\end{theorem}

The next theorem gives a compact bound for $|F_d(p)(t)|$.

\begin{theorem} \label{thm: bounds of F_d(p)(t)}
  Let $d \geq 2$ and $p$ be a real polynomial. Then
  \begin{displaymath}
  \max_{  t \in [-1, 1]} \abs{F_d(p)(t)} \leq \max_{t \in [a, b]} \abs{p(t)}.
  \end{displaymath}
\end{theorem}

\begin{proof}
  For $t\in [-1,1]$, let $t=\cos\theta$ for $\theta\in [0,\pi]$. Define
  $q_d(\theta) = F_d(p)(\cos \theta)$ and $g(\theta) = p(\cos \theta) h(\cos \theta)$.
  Then by Lemma 1.4 of \cite[section 1.1.2]{ApproxIntroduction}, we have
  \begin{displaymath}
    q_d(\theta) = \frac{1}{\pi} \int_{-\pi}^\pi g(\theta + \phi) u_d(\phi) {\rm d} \phi
  \end{displaymath}
  with $u_d(\phi) = \frac{1}{2} + \sum_{j=1}^d \rho_{j, d} \cos(j \phi)$.
  The proof of Theorem 3.1 in \cite{CJ_FEAST_cross} has shown
  \begin{displaymath}
    u_d(\phi) \geq 0, \quad \frac{1}{\pi} \int_{-\pi}^\pi u_d(\phi) {\rm d} \phi = 1.
  \end{displaymath}
  Thus, for $t \in [-1, 1]$ and $\theta = \arccos(t)$, we obtain
  \begin{align*}
    \abs{F_d(p)(t)}
    = \abs{q_d(\theta)}
    &\leq \frac{1}{\pi} \int_{-\pi}^\pi \abs{g(\theta + \phi)} u_d(\phi) {\rm d} \phi \\
    &\leq \max_{s \in [a, b]} {\abs{p(s)}} \cdot \frac{1}{\pi} \int_{-\pi}^\pi u_d(\phi) {\rm d} \phi
    = \max_{s \in [a, b]} \abs{p(s)}.
  \end{align*}
\end{proof}

\subsection{The CJ--SS--RR and CJ--SS--RRR methods}\label{subsec:2}

Given positive integers $M, \ell$ and a full column rank matrix $V \in \Complex^{n \times \ell}$, define
\begin{equation} \label{eq: def: polynomial approximation of S}
  S =(S_0,\ldots,S_{M-1})=: \left( F_d(p_0)(A) V, \dots, F_d(p_{M-1})(A) V \right),
\end{equation}
where $p_j(t)$'s are shifted-and-scaled Chebyshev polynomials of degree $j$ of the first kind:
\begin{equation} \label{eq: def: p_j(t)}
  p_j(t) = T_j\left( \frac{2t - a - b}{b - a} \right), \quad j = 0, 1, \dots, M-1.
\end{equation}
Notice that $\Range(S)$ is the search subspace used in the CJ--SS--RR and CJ--SS--RRR methods \cite{CJ_SS_RRR} and that $|p_j(t)|\leq 1$ for $t\in [a,b]$. One should choose $\ell$ no smaller
than the multiplicity of each distinct eigenvalue in $[a,b]$; see \cite{jia_block}.
We remark that, in the contour integral-based SS--RR and SS--RRR methods,
$S_i$, $i=0,\ldots, M-1$, are generated by a contour integral-based numerical
quadrature rule, e.g., the trapezoidal rule, where the
monomial moments and the better shifted-and-scaled Chebyshev moments
are used in \cite{block_SS_RR,SS_RR} and in \cite{CJ_SS_RR,CJ_SS_RRR}, respectively.


Let the columns of $U$ form an orthonormal basis of $\Range(S)$. Then
the CJ--SS--RR and CJ--SS--RRR methods compute approximations to some of the eigenpairs of $A$
by the Rayleigh--Ritz procedure and the refined Rayleigh--Ritz
procedure \cite{bai2000,Parlett_Symmetric,Stewart_Eigen,book:VanDerVorst2002}
as follows:
\begin{align}
  \label{eq: def: Rayleigh--Ritz}
  & U^H A U y = \tilde{\lambda} y, \quad \tilde{x} = U y, \quad \text{with } \norm{y} = 1, \\
  \label{eq: def: refined vector}
  & \hat{x} = \mathop{\arg\min}_{w \in \Range(S), \norm{w} = 1} \norm{ (A - \tilde{\lambda} I) w },  \quad \hat{\lambda} = \hat{x}^H A \hat{x},
\end{align}
respectively, where $(\tilde{\lambda}, \tilde{x})$ and $(\hat{\lambda}, \hat{x})$
are called the Ritz pairs and refined Ritz pairs, respectively. For the efficient and accurate
computation of $\hat x$, see \cite{jia1997refined, jia2000refined,Stewart_Eigen}; for a general
search subspace, the authors in \cite{CJ_SS_RRR} have recently proposed
a more efficient procedure to compute $\hat x$ accurately.

As has been argued and elaborated in \cite{CJ_SS_RRR}, for $M\ell >n_{ev}$,
there may be spurious Ritz values in $[a,b]$; that is, there may be more Ritz values  in $[a,b]$
that either
approximate the same eigenvalue counting its multiplicity
or do {\em not} approximate any eigenvalue. In both cases, the CJ--SS--RR
method may encounter severe difficulties. But the CJ--SS--RRR method can settle these two issues
perfectly: Unlike the Ritz pairs obtained by the CJ--SS--RR method,
measured by the residual norms of refined Ritz pairs, once the projection subspace
becomes reasonably good, if there are more refined Ritz vectors that either
converge or do {\em not} converge, then in the first case
the refined Ritz vectors are nearly linearly dependent and in the second case
the residual norms must not be small.
The authors propose an effective and efficient refined removal approach in \cite{CJ_SS_RRR} to remove
extra refined Ritz vectors and thus spurious refined Ritz values, so that the CJ--SS--RRR method with the refined removal approach can compute the $n_{ev}$ desired eigenpairs reliably. The same
removal approach applies to the CJ--SS--RR method, but may not work well because
of the intrinsic convergence deficiency of the Rayleigh--Ritz procedure; see
\cite{CJ_SS_RRR} for details.

Specifically, the refined removal approach consists of two steps.
First,
two refined Ritz values $\hat{\lambda}_i$ and $\hat{\lambda}_j$ in $[a,b]$
are proved to be in the same cluster if
\begin{equation} \label{eq: criterion for clustering refined Ritz values}
  \abs{\hat{\lambda}_i - \hat{\lambda}_j}
  \leq \norm{\hat{r}_i} + \norm{\hat{r}_j},
\end{equation}
where $\hat{r}_i = A \hat{x}_i - \hat{\lambda}_i \hat{x}_i$, provided that $\norm{\hat{r}_i}$
and $\norm{\hat{r}_j}$ are reasonably small.
Second,
for a cluster of $s$ refined Ritz values $\hat{\lambda}_{i_1}, \hat{\lambda}_{i_2}, \dots,
\hat{\lambda}_{i_s}$, write the corresponding refined
Ritz vectors $\hat{x}_{i_1}, \hat{x}_{i_2}, \dots, \hat{x}_{i_s}$.
Then if
\begin{equation}\label{xhat}
\hat{X}_i = (\hat{x}_{i_1}, \hat{x}_{i_2}, \dots, \hat{x}_{i_s})
\end{equation}
has $s-s_1$ singular values
{\em smaller} than
\begin{equation} \label{eq: criterion for removing refined Ritz vectors}
  \sqrt{s} \max_{1 \leq t \leq s} \frac{\norm{\hat{r}_{i_t}}}{\tilde{g}_{i_t}},
\end{equation}
where $\tilde{g}_{i_t} = \min_{\substack{j \not\in \{i_1, \dots, i_s\}}} \abs{\hat{\lambda}_{j} - \hat{\lambda}_{i_t}}$, the numerical rank of $\hat{X}_i$ is $s_1$ and
there are $s-s_1$ spurious refined Ritz values
and extra refined Ritz vectors in the cluster. The refined removal approach is outlined in \cref{alg: removal}; replacing the refined Ritz values,
the refined Ritz vectors and the corresponding residuals in the algorithm, we get the removal approach for the
CJ--SS--RR method. We refer to \cite{CJ_SS_RRR} for theory and computation.
In the sequel, whenever speaking of CJ--SS--RRR, we mean that it has been combined with
the refined removal approach; for brevity and the sake of efficiency and reliability, we will only
focus on CJ--SS--RRR.

\begin{algorithm}[htbp]
\caption{Refined removal approach}
\label{alg: removal}
\begin{algorithmic}[1]
    \STATE
    Cluster the refined Ritz values based on
    \cref{eq: criterion for clustering refined Ritz values}.

 \STATE {
  For each cluster,
  compute the SVD of $\hat{X}_i$ defined by \eqref{xhat},
  and check how many of its singular values are smaller than the quantity in
  \cref{eq: criterion for removing refined Ritz vectors};
  retain $s_1$ refined Ritz values and vectors as genuine, and remove the extra $s-s_1$ refined Ritz vectors.
  }
\end{algorithmic}
\end{algorithm}


\section{The restart schemes}
\label{sec: restart strategies}
Given a starting matrix $V^{(0)} \in \Complex^{n \times \ell}$,
the subspace $\mathcal{R}(S^{(k)})$
in the $k$-th restart of the CJ--SS--RRR method is the column space of
\begin{equation} \label{eq: def: k-th search subspace S^(k)}
  S^{(k)} = \left( F_d(p_0)(A)V^{(k)}, \dots, F_d(p_{M-1})(A)V^{(k)} \right), k=0,1,\ldots,
\end{equation}
where $p_j(t)$'s with $p_0(t)=1$ are the polynomials defined in \cref{eq: def: p_j(t)}. Different choices of $V^{(k)}$ with $k\geq 1$ lead to different restart schemes. We focus on reasonable and effective
choices of $V^{(k)}$ in this section, establish rigorous theoretical results on
each of them, get insight into
their numerical behavior, and highlight their remarkable numerical effects on restarted algorithms.

\subsection{The SI restart or natural restart}
\label{subsec: SI restart}

The restart scheme proposed by Sakurai, Futamura, and Tadano \cite{SS_restart1} and applied to
our context reads as
\begin{equation} \label{eq: def: SI restart}
  V^{(k+1)} = F_d(1)(A) V^{(k)},
\end{equation}
which is the first $\ell$ columns of $S^{(k)}$ due to $p_0(t) \equiv 1$ in \cref{eq: def: p_j(t)}.
In the SS-type methods, $F_d(1)(A)$ is an approximation to the zeroth moment, i.e.,
the (generalized) spectral projector of the eigenvalues of $A$ in $[a,b]$,
obtained by a numerical quadrature \cite{CJ_FEAST_augmented,CJ_FEAST_GSVD}.

A CJ--SS--RRR algorithm with such restart scheme is described in \cref{alg: CJSSRRR}.
In step \ref{step: CJSSRRR: convergence test}, a refined Ritz pair $(\hat{\lambda}_i^{(k)}, \hat{x}_i^{(k)})$
is said to have converged if its relative residual norm drops below a prescribed stopping
tolerance $tol$:
\begin{equation} \label{eq: convergence criterion}
  \frac{\norm{\hat{r}_i^{(k)}}}{\norm{A}} < tol.
\end{equation}

\begin{algorithm}[htbp]
\caption{The restarted CJ--SS--RRR algorithm}
\label{alg: CJSSRRR}
\begin{algorithmic}[1]
  \REQUIRE The Hermitian matrix $A$, the real interval $[a, b]$, the series degree $d$, the moment number $M$, and a column orthonormal matrix $V^{(0)}\in \mathbb{C}^{n\times \ell}$ with $M\ell \geq n_{ev}$.
  \STATE Compute the CJ coefficients by \cref{eq: def: Chebyshev coefficient,eq: def: Jackson damping factor}.
  \FOR {$k = 0, 1, 2, \ldots$}
    \STATE Compute $S^{(k)}$ by \cref{eq: def: k-th search subspace S^(k)}.
    \STATE Compute the QR decomposition: $S^{(k)} = U^{(k)}R^{(k)}$.
    \STATE Use the Rayleigh--Ritz procedure \cref{eq: def: Rayleigh--Ritz} to
     compute the Ritz values $\tilde{\lambda}_i^{(k)}, i = 1, 2, \ldots, M\ell$. \label{step: CJSSRRR: RR}
    \STATE Select all
    the Ritz values $\tilde{\lambda}_1^{(k)}, \ldots, \tilde{\lambda}_p^{(k)}$ in $[a, b]$, and
    compute their corresponding refined Ritz vectors $\hat{x}_1^{(k)}, \ldots, \hat{x}_p^{(k)}$ by \cref{eq: def: refined vector}. \label{step: CJSSRRR: refined}
    \STATE Compute the refined Ritz values $\hat{\lambda}_1^{(k)}, \ldots, \hat{\lambda}_p^{(k)}$ and the residuals $\hat{r}_1^{(k)}, \ldots, \hat{r}_p^{(k)}$.
    \STATE Remove spurious refined Ritz values among $\hat{\lambda}_1^{(k)}, \ldots,
    \hat{\lambda}_p^{(k)}$ and
    extra refined Ritz vectors by \cref{alg: removal}. \label{step: CJSSRRR: removal}
    \STATE Test the convergence of the retained refined Ritz pairs. \label{step: CJSSRRR: convergence test}
    \STATE If not converged, update the starting $V^{(k+1)} = F_d(1)(A) V^{(k)}$, the first $\ell$ columns of $S^{(k)}$, and restart. \label{step: CJSSRRR: restart}
  \ENDFOR
  \RETURN The converged $n_{ev}$ refined Ritz pairs $(\hat{\lambda}_i^{(k)}, \hat{x}_i^{(k)})$.
\end{algorithmic}
\end{algorithm}

It is shown in \cite{SS_theory2016,SS_review} that
the restart scheme \eqref{eq: def: SI restart}
performs one iteration of subspace iteration on the matrix $F_d(1)(A)$
applied to $\Range({S}^{(k)})$ because
\begin{align}
  S^{(k+1)} &= \left( F_d(1)(A) V^{(k+1)}, \dots, F_d(p_{M-1})(A) V^{(k+1)} \right) \nonumber\\
  &= F_d(1)(A) \left( F_d(1)(A) V^{(k)}, \dots, F_d(p_{M-1})(A) V^{(k)} \right) \nonumber\\
  &= F_d(1)(A) S^{(k)}, \label{onerestart}
\end{align}
i.e.,
\begin{equation} \label{eq: SI property of S^(k)}
  \Range(S^{(k+1)}) = F_d(1)(A) \cdot \Range(S^{(k)}).
\end{equation}
Therefore, the convergence results on subspace iteration \cite{Golub_Matrix,Parlett_Symmetric,Saad_Eigen,Stewart_Eigen} apply to
\cref{alg: CJSSRRR}; see \cite{SS_theory2016,SS_review}.
We call \eqref{eq: def: SI restart} the \emph{naive subspace iteration (SI) restart} or {\em naive natural restart}.

However, in finite precision arithmetic,
the naive SI restart \eqref{eq: def: SI restart} may suffer from
severe numerical instability: Denote $P=F_d(1)(A)$.
Then it is inductive from \eqref{onerestart} that
\cref{alg: CJSSRRR} computes $S^{(k)} = P^k S^{(0)}$ at restart $k$.
Unfortunately, as is easily seen and already
stated in \cite{Parlett_Symmetric}, the columns of
$P^k S^{(0)}$ are a bad basis of $\Range(P^k S^{(0)})$ because
each column of $P^k S^{(0)}$ converges to the eigenspace
associated with the dominant eigenvalue of $P$ as $k$ increases. In other words,
the condition number of $S^{(k)}$ will tend to infinity as $k$ increases.
Consequently, the {\em computed} subspace generated by the factor $U^{(k)}$ at step 4 of
\cref{alg: CJSSRRR}, which is the Q-factor of the {\em computed}
QR decomposition of highly ill-conditioned $S^{(k)}$, is far away from
the true subspace $\Range(S^{(k)})$ \cite[Chapter 19]{higham} and ultimately bears no relation
to $\Range(S^{(k)})$. As a result,
\cref{alg: CJSSRRR} should be numerically unstable. We will
rigorously prove quantitative results on the behavior of $S^{(k)}$,
and gain more insight into severe consequences these cause.


As a remedy, suppose that $P V^{(k)}$ is of full column rank. Then instead of using $V^{(k+1)}=PV^{(k)}$
in \eqref{onerestart},
we take the column orthonormal matrix $V^{(k+1)}$ by computing the compact QR decomposition
\begin{equation}\label{eq: def: SI-orth restart}
P V^{(k)} = V^{(k+1)} R,
\end{equation}
where $R$ is nonsingular upper triangular.
We call $V^{(k+1)}$ in \cref{eq: def: SI-orth restart}
the \emph{orthonormalized SI restart} or simply {\em natural restart},
and construct
\begin{align*}
  S^{(k+1)}
  &= \left( F_d(1)(A) V^{(k+1)}, \dots, F_d(p_{M-1})(A) V^{(k+1)} \right) \\
  &= F_d(1)(A) \left( F_d(1)(A) V^{(k)} R^{-1}, \dots, F_d(p_{M-1})(A) V^{(k)} R^{-1} \right)  \\
  &= F_d(1)(A) S^{(k)} \left( I \otimes R \right)^{-1}=P S^{(k)} \left( I \otimes R \right)^{-1},
\end{align*}
where $\otimes$ denotes the Kronecker product. Therefore,
relation \cref{eq: SI property of S^(k)} still holds,
and each natural restart performs one iteration of subspace iteration on $P$.
Nevertheless, in finite precision arithmetic,
the numerical performance of the natural
restart \eqref{eq: def: SI-orth restart} is expected to be better than that of the naive
restart \eqref{onerestart} since
the former $V^{(k+1)}$ is column orthonormal and the latter is not and may be ill-conditioned.
Yet we will later on prove that,
for both the natural restarts \eqref{eq: def: SI restart}
and \eqref{eq: def: SI-orth restart}, it always holds that
\begin{equation}\label{infbound}
  \lim_{k \to \infty} \frac{\sigma_1(S^{(k)})}{\sigma_{n_{ev}}(S^{(k)})} = \infty.
\end{equation}

\begin{remark}\label{rem3.1}
Relation \eqref{infbound} indicates that the numerical rank of $S^{(k)}$ must be smaller
than $n_{ev}$ when $k$ is large enough. Therefore, in finite precision arithmetic, given a dropping
threshold, for sufficiently
large $k$, the effective dimension of $\Range(S^{(k)})$
is ultimately smaller than $n_{ev}$ when using the truncated singular value decomposition (TSVD) or
rank revealing QR decomposition of $S^{(k)}$
to determine the numerical rank and obtain an effective subspace with the
dimension being the numerical rank, as commonly done in the SS--type methods.
In this case, the restarted SS-type methods can compute only {\em partial rather than all}
the $n_{ev}$ desired eigenpairs. On the other hand, if one simply computes
the QR decomposition of $S^{(k)}$ and does not reduce the subspace dimension,
the computed subspace is ultimately
far from the true $\Range(S^{(k)})$ because of numerical rank deficiency of $S^{(k)}$.
As a disastrous consequence, if greater accuracy is required, 
\cref{alg: CJSSRRR}
fails to extract $n_{ev}$ eigenpairs from $\Range(S^{(k)})$,
so do the other SS-type algorithms using the same natural restart.
We will numerically justify these phenomena in \cref{sec: numerical experiments}.
\end{remark}

In what follows we analyze the natural restart scheme and prove \eqref{infbound} quantitatively.
Before proceeding, 
we need some notation and definitions. Denote the spectral decomposition of $A$ by
\begin{displaymath}
  A = X \Lambda X^H,
\end{displaymath}
where $X = (x_1, x_2, \dots, x_n) \in \Complex^{n \times n}$ is unitary and $\Lambda = \Diag\{ \lambda_1,
\lambda_2, \dots, \lambda_n \} \in \Real^{n \times n}$ is diagonal.
Denote by $\gamma_i =F_d(1)(\lambda_i)$, $i = 1, 2, \dots, n$
the eigenvalues of $P=F_d(1)(A)$, and assume that the degree $d$ of the CJ series expansion satisfies \cref{eq: theoretical
requirement on degree}.
Then by \cref{thm: eigenvalues of F_d(1)(A)},
we label the eigenvalues of $P$ in decreasing order
\begin{equation} \label{eq: eigenvalues of F_d(1)(A)}
  1 \geq \gamma_1 \geq \gamma_2 \geq \dots \geq \gamma_{n_{ev}} > \frac{1}{4} > \gamma_{n_{ev}+1} \geq \dots \geq \gamma_n \geq 0,
\end{equation}
so that the desired eigenvalues $\lambda_1, \lambda_2, \dots, \lambda_{n_{ev}}$ of $A$ in $[a,b]$
correspond to the $n_{ev}$ dominant eigenvalues of $P$ and the associated eigenvectors are
$x_1,x_2,\ldots,x_{n_{ev}}$. Write
\begin{equation} \label{partition}
  \begin{aligned}
    & X_i = (x_1, x_2, \dots, x_i), \ \Lambda_i = \Diag\{ \lambda_1, \lambda_2, \dots, \lambda_i \}, \ \Gamma_i = \Diag\{ \gamma_1, \gamma_2, \dots, \gamma_i \}, \\
    & X_{i,\perp} = (x_{i+1}, \dots, x_n), \ \Lambda_{i,\perp} = \Diag\{ \lambda_{i+1}, \dots, \lambda_n \}, \ \Gamma_{i,\perp} = \Diag\{ \gamma_{i+1}, \dots, \gamma_n \}.
  \end{aligned}
\end{equation}

The following definition of canonical angles between two subspaces with possibly unequal dimensions can be
found in \cite{Stewart_Eigen}.

\begin{definition} \label{def: canonical angles between subspaces}
  Let $\mathcal{X}$, $\mathcal{Y}$ be two subspaces of $\Complex^n$ with $\dim \mathcal{X} \geq \dim \mathcal{Y} = \ell > 0$,
  and the columns of $X$ and $Y$ be orthonormal bases of $\mathcal{X}$ and $\mathcal{Y}$, respectively.
  Then the $\ell$ canonical angles between $\mathcal{X}$ and $\mathcal{Y}$ are
  \begin{displaymath}
    \theta_i(\mathcal{X}, \mathcal{Y}) = \arccos \sigma_{\ell-i+1}(X^H Y), \quad i = 1, 2, \dots, \ell.
  \end{displaymath}
  In particular, write the largest canonical angle
  \begin{displaymath}
    \angle(\mathcal{X}, \mathcal{Y}) := \theta_1(\mathcal{X}, \mathcal{Y}) = \arccos \sigma_{\min}(X^H Y).
  \end{displaymath}
\end{definition}

For $\angle(\mathcal{X}, \mathcal{Y})$, we have the following proposition,
which is Theorem 2.2 in \cite[Chapter 4]{Stewart_Eigen} and can also be retrieved from~\cite{jiastewart2001}.

\begin{proposition} \label{prop: sin angle between subspaces}
  Let $\mathcal{X}$ and $\mathcal{Y}$ be two subspaces of $\Complex^n$ with $\dim \mathcal{X} \geq \dim \mathcal{Y} > 0$,
  and the columns of $X$ and $Y$ be their orthonormal bases.
  Denote by the columns of $X_\perp$ an orthonormal basis of the orthogonal complement of $\mathcal{X}$ with
  respect to $\mathbb{C}^n$.
  Then
  \begin{equation} \label{eq: def: sin angle between subspaces}
    \sin \angle(\mathcal{X}, \mathcal{Y}) = \norm{ X_\perp^H Y }.
  \end{equation}
\end{proposition}


It is important to note from
\cite{jiastewart2001} that, for two subspaces of unequal dimensions, the measure
$ \sin \angle(\mathcal{X}, \mathcal{Y})$ is not symmetric in the arguments $\mathcal{X}$ and
$\mathcal{Y}$, i.e.,  $\sin \angle(\mathcal{X}, \mathcal{Y})\not= \sin \angle(\mathcal{Y}, \mathcal{X})$
generally. In fact, for $\dim \mathcal{X} > \dim \mathcal{Y} > 0$, we have
$\sin \angle(\mathcal{Y}, \mathcal{X})\equiv 1$ but
$\sin \angle(\mathcal{X}, \mathcal{Y})<1$ provided that no vector in $\mathcal{Y}$ is orthogonal to
$\mathcal{X}$. In contrast, by definition, $\cos \angle(\mathcal{X}, \mathcal{Y})= \cos \angle(\mathcal{Y}, \mathcal{X})$ always, independently of $\dim\mathcal{X}$ and $\dim\mathcal{Y}$.


\begin{lemma} \label{lem: norm bounds of S}
  For $S = (S_0, \dots, S_{M-1}) \in \Complex^{n \times M\ell}$ with $S_j \in \Complex^{n \times \ell}$, we have
  \begin{equation}\label{lowupbound}
    \max_{0 \leq j \leq M-1} \norm{S_j} \leq \norm{S}
    \leq \Big( \sum_{j=0}^{M-1}\norm{S_j}^2 \Big)^{1/2}.
  \end{equation}
\end{lemma}

\begin{proof}
  Let $y_j \in \Complex^\ell \setminus \{0\}$ satisfy $\norm{S_j y_j} = \norm{S_j} \norm{y_j}$ for $j = 0, 1, \dots, M-1$.
  Then
  \begin{displaymath}
    \norm{S}
    \geq \max_{0\leq j \leq M-1} \frac{ \norm{S \cdot (0,\dots,y_j^H,\dots,0)^H} }{ \norm{(0,\dots,y_j^H,\dots,0)^H} }
    = \max_{0\leq j \leq M-1} \frac{ \norm{S_jy_j} }{ \norm{y_j} }
    = \max_{0\leq j \leq M-1} \norm{S_j},
    \end{displaymath}
  which establishes the lower bound in \eqref{lowupbound}.

  For any $u = (u_0^H, \dots, u_{M-1}^H)^H \in \Complex^{M\ell} \setminus \{ 0 \}$ with $u_j \in \Complex^\ell$, we have
  \begin{align*}
    \norm{S u}^2
    &= \norm{ \sum_{j=0}^{M-1} S_j u_j }^2
    \leq \Big( \sum_{j=0}^{M-1} \norm{S_j u_j} \Big)^2
    \leq \Big( \sum_{j=0}^{M-1} \norm{S_j} \norm{u_j} \Big)^2 \\
    &\leq \Big( \sum_{j=0}^{M-1} \norm{S_j}^2 \Big) \Big( \sum_{j=0}^{M-1} \norm{u_j}^2 \Big)
    = \Big( \sum_{j=0}^{M-1} \norm{S_j}^2 \Big) \norm{u}^2,
  \end{align*}
  from which the upper bound in \eqref{lowupbound} follows.
\end{proof}

We now present one of the main results in this subsection.

\begin{theorem} \label{thm: ratio of singular values of S^(k) for SI-like}
  Let the CJ series degree $d$ satisfy \cref{eq: theoretical requirement on degree},
  and label the eigenvalues $\gamma_i$ of $P$ as in \cref{eq: eigenvalues of F_d(1)(A)}.
  Assume that
  \begin{displaymath}
    1 < n_{ev} \leq M \ell, \quad \text{and} \quad \gamma_1 > \gamma_{n_{ev}}.
  \end{displaymath}
  For a column orthonormal matrix $V^{(0)} \in \Complex^{n \times \ell}$,
  denote by $\mathcal{V}^{(0)}={\rm span}\{V^{(0)}\}$,
  suppose $x_1^H V^{(0)} \neq 0$, and
  let $S^{(k)}$ in \eqref{eq: def: k-th search subspace S^(k)} be generated by $V^{(k)}$
  in \cref{eq: def: SI restart}.
  Then
  \begin{equation} \label{eq: ratio of singular values of S^(k) for SI-like}
    \frac{\sigma_1(S^{(k)})}{\sigma_{n_{ev}}(S^{(k)})}
    \geq \frac{\gamma_1}{\sqrt{M}} \left( \frac{\gamma_1}{\gamma_{n_{ev}}} \right)^k
    \frac{\cos\angle(\mathcal{V}^{(0)},x_1)}{\sin\angle(x_1, \mathcal{V}^{(0)})},\ k\geq 1.
  \end{equation}
\end{theorem}

\begin{proof}
For convenience, with \eqref{partition} we denote
\begin{equation} \label{eq: def: tilde_perp}
  \widetilde{X}_\perp = X_{n_{ev}-1,\perp}, \quad \widetilde{\Lambda}_\perp = \Lambda_{n_{ev}-1,\perp}, \quad \widetilde{\Gamma}_\perp = \Gamma_{n_{ev}-1,\perp}.
\end{equation}
It is known from the Eckart--Young theorem (cf. \cite[Theorem 2.4.8]{Golub_Matrix}) that
\begin{displaymath}
  \sigma_{n_{ev}}(S^{(k)}) = \min_{\Rank(T) < n_{ev}} \norm{ S^{(k)} - T }.
\end{displaymath}
Take $T = X_{n_{ev}-1} X_{n_{ev}-1}^H S^{(k)}$.
Then
\begin{equation} \label{eq: original upper bound of n_ev-th singular value of S^(k)}
    \sigma_{n_{ev}}(S^{(k)})
    \leq \norm{ S^{(k)} - X_{n_{ev}-1} X_{n_{ev}-1}^H S^{(k)} }
    = \norm{ \widetilde{X}_\perp \widetilde{X}_\perp^H S^{(k)} }
    = \norm{ \widetilde{X}_\perp^H S^{(k)} }.
\end{equation}
From \cref{eq: def: k-th search subspace S^(k)}, \cref{lem: norm bounds of S}
and $V^{(k)}=P^kV^{(0)}$, we obtain
\begin{equation} \label{eq: upper bound of norm of X_perp^H S^(k)}
  \begin{aligned}
    \norm{\widetilde{X}_\perp^H S^{(k)}}^2
    & = \norm{ \widetilde{X}_\perp^H (F_d(p_0)(A) P^k V^{(0)}, \dots, F_d(p_{M-1})(A) P^k V^{(0)}) }^2 \\
    &\leq \sum_{j=0}^{M-1} \norm{ F_d(p_j)(\widetilde{\Lambda}_\perp) \widetilde{\Gamma}_\perp^k \widetilde{X}_\perp^H V^{(0)} }^2 \\
    &\leq \gamma_{n_{ev}}^{2k} \norm{\widetilde{X}_\perp^H V^{(0)}}^2 \sum_{j=0}^{M-1} \norm{ F_d(p_j)(\widetilde{\Lambda}_\perp) }^2.
  \end{aligned}
\end{equation}
By \cref{thm: bounds of F_d(p)(t)} and \cref{eq: def: p_j(t)}, for $j = 0, 1, \dots, M-1$, we have
\begin{displaymath}
  \norm{ F_d(p_j)(\widetilde{\Lambda}_\perp) }
  \leq \max_{t\in [-1, 1]}|F_d(p_j)(t)|
  \leq \max_{t\in [a, b]}|p_j(t)| = 1.
\end{displaymath}
Therefore,
\begin{equation} \label{eq: sum of norms of F_d(p_j)(Lambda_perp)}
  \sum_{j=0}^{M-1} \norm{ F_d(p_j)(\widetilde{\Lambda}_\perp) }^2 \leq M.
\end{equation}
Notice that $\widetilde{X}_\perp$ is a subset of columns of $X_{1,\perp}$.
Substituting \cref{eq: upper bound of norm of X_perp^H S^(k),eq: sum of norms of F_d(p_j)(Lambda_perp)} into \cref{eq: original upper bound of n_ev-th singular value of S^(k)} yields
\begin{displaymath}
  \sigma_{n_{ev}}(S^{(k)})
  \leq \gamma_{n_{ev}}^k \sqrt{M} \norm{\widetilde{X}_\perp^H V^{(0)}}
  \leq \gamma_{n_{ev}}^k \sqrt{M} \norm{X_{1,\perp}^H V^{(0)}}=\gamma_{n_{ev}}^k \sqrt{M}\sin\angle(x_1, \mathcal{V}^{(0)}).
\end{displaymath}
On the other hand, by \cref{lem: norm bounds of S} and $P=F_d(1)(A)=F_d(p_0)(A)$, we have
\begin{displaymath}
  \sigma_1(S^{(k)})
  \geq \norm{ x_1^H S^{(k)} }
  \geq \norm{ x_1^H P^{k+1} V^{(0)} }
  = \gamma_1^{k+1} \norm{x_1^H V^{(0)}}= \gamma_1^{k+1} \cos\angle(\mathcal{V}^{(0)},x_1).
\end{displaymath}
Combining the above two relations establishes \cref{eq: ratio of singular values of S^(k) for SI-like}.
\end{proof}



\begin{remark}
For the lower bound \eqref{eq: ratio of singular values of S^(k) for SI-like},
we comment that, by \cref{prop: sin angle between subspaces} and the comments that follow, it holds that
$$
\sin\angle(x_1,\mathcal{V}^{(0)}) =\sin\angle(\mathcal{V}^{(0)},x_1)
$$ for $\ell=1$ and $\sin\angle(x_1,\mathcal{V}^{(0)})\equiv 1$ for $\ell>1$, so that
$\frac{\cos\angle(\mathcal{V}^{(0)},x_1)}{\sin\angle(x_1, \mathcal{V}^{(0)})}=\cot\angle(\mathcal{V}^{(0)},x_1)$ for $\ell=1$ and $\frac{\cos\angle(\mathcal{V}^{(0)},x_1)}{\sin\angle(x_1, \mathcal{V}^{(0)})}=\cos\angle(\mathcal{V}^{(0)},x_1)$ for $\ell>1$.
\end{remark}

Below we present a smaller lower bound for the ratio
$\sigma_1(S^{(k)}) / \sigma_{n_{ev}}(S^{(k)})$ obtained by
the natural restart \cref{eq: def: SI-orth restart}, showing that it tends to infinity
as $k$ increases.

\begin{theorem} \label{thm: ratio of singular values of S^(k) for SI-orth}
  Let the CJ series degree $d$ satisfy \cref{eq: theoretical requirement on degree},
  and label the eigenvalues $\gamma_i$ of $P$ as in \cref{eq: eigenvalues of F_d(1)(A)},
  and assume that
  \begin{displaymath}
    1 \leq \ell < n_{ev} \leq M \ell, \quad \text{and} \quad \gamma_{\ell} > \gamma_{n_{ev}}.
  \end{displaymath}
  For a column orthonormal matrix $V^{(0)} \in \Complex^{n \times \ell}$,
  suppose that $X_\ell^H V^{(0)}$ is nonsingular with $X_\ell$ defined by \eqref{partition},
  and write $\mathcal{X}_\ell = \Range(X_\ell)$ and $\mathcal{V}^{(0)} = \Range(V^{(0)})$.
 Let $S^{(k)}$ in \cref{eq: def: k-th search subspace S^(k)} be generated by $V^{(k)}$
 in \eqref{eq: def: SI-orth restart} with $k$ replaced by $k-1$.
  Then
  \begin{equation} \label{eq: ratio of singular values of S^(k) for SI-orth}
    \frac{\sigma_1(S^{(k)})}{\sigma_{n_{ev}}(S^{(k)})}
    \geq \frac{\gamma_\ell}{\sqrt{M}}
      \left( \frac{\gamma_\ell}{\gamma_{n_{ev}}} \right)^k
      \cot \angle(\mathcal{V}^{(0)},\mathcal{X}_\ell)\cos \angle(\mathcal{V}^{(0)},\mathcal{X}_\ell),\ k\geq 1.
  \end{equation}
\end{theorem}

\begin{proof}
We first establish an upper bound for $\sigma_{n_{ev}}(S^{(k)})$, showing how fast it
tends to zero at least as $k$ increases.
With \eqref{partition}, write $\widetilde{X}_\perp$, $\widetilde{\Lambda}_\perp$ and $\widetilde{\Gamma}_\perp$ as in \cref{eq: def: tilde_perp}.
Then we can decompose $V^{(0)}$ into mutually orthogonal sums:
\begin{align*}
  V^{(0)} &= X_{n_{ev}-1} (X_{n_{ev}-1}^H V^{(0)}) + \widetilde{X}_\perp (\widetilde{X}_\perp^H V^{(0)})\\
 & =: X_\ell (X_\ell^H V^{(0)}) + W+\widetilde{X}_\perp (\widetilde{X}_\perp^H V^{(0)}),
\end{align*}
where $W^H X_{\ell}=0$ and $W^H\widetilde{X}_\perp=0$.
Therefore,
\begin{equation} \label{eq: SI-orth: expression of P^k V}
  P^k V^{(0)} = X_\ell \Gamma_\ell^k (X_\ell^H V^{(0)}) + \widetilde{X}_\perp \widetilde{\Gamma}_\perp^k (\widetilde{X}_\perp^H V^{(0)}) + P^k W.
\end{equation}
Let
\begin{equation} \label{eq: SI-orth: def: Omega_k}
  \Psi_k = \Gamma_\ell^k (X_\ell^H V^{(0)}), \quad
  \Omega_k = \widetilde{\Gamma}_\perp^k (\widetilde{X}_\perp^H V^{(0)}) \Psi_k^{-1}.
\end{equation}
Then we can write \cref{eq: SI-orth: expression of P^k V} as
\begin{equation} \label{eq: SI-orth: expression of P^k V Psi_k^{-1}}
  P^k V^{(0)} \Psi_k^{-1} = X_\ell + \widetilde{X}_\perp \Omega_k + P^k W \Psi_k^{-1},
\end{equation}
where the three terms in the right-hand side are mutually orthogonal since $P$ and $A$
share the same eigenvectors.

Since the columns of $V^{(k)}$ are an orthonormal basis of $P^k \mathcal{V}^{(0)}$, we obtain
\begin{equation} \label{eq: SI-orth: norm of X_perp^H V^(k)}
  \begin{aligned}
    \norm{\widetilde{X}_\perp^H V^{(k)}}
    &= \max_{y \in \Complex^\ell, \norm{y} = 1} \norm{\widetilde{X}_\perp^H V^{(k)} y}
    = \max_{u \in P^k \mathcal{V}^{(0)}, \norm{u} = 1} \norm{\widetilde{X}_\perp^H u} \\
    &= \max_{u \in P^k \mathcal{V}^{(0)}, u \neq 0} \sin \angle(\Range(X_{n_{ev}-1}), u).
  \end{aligned}
\end{equation}
Write $u = P^k V^{(0)} \Psi_k^{-1} z$.
Then by \eqref{eq: SI-orth: def: Omega_k}, \cref{eq: SI-orth: expression of P^k V Psi_k^{-1}} and $X_{n_{ev}-1}^H\widetilde{X}_\perp=0$, we have
\begin{displaymath}
  |\tan \angle(\Range(X_{n_{ev}-1}), u)|
  = \frac{\norm{\widetilde{X}_\perp^H u}}{\norm{X_{n_{ev}-1}^H u}}
  = \frac{\norm{\Omega_k z}}{\norm{X_{n_{ev}-1}^H (X_\ell + P^k W \Psi_k^{-1}) z}}.
\end{displaymath}
Since $X_\ell^H P^kW=0$ and $\mathcal{X}_\ell\oplus \Range(P^kW) \subseteq \Range(X_{n_{ev}-1})$, we obtain
\begin{align*}
  \norm{X_{n_{ev}-1}^H (X_\ell + P^k W \Psi_k^{-1}) z}^2
  &= \norm{X_\ell z + P^k W \Psi_k^{-1} z}^2 \\
  &= \norm{X_\ell z}^2 + \norm{P^k W \Psi_k^{-1} z}^2
  \geq \norm{X_\ell z}^2 = \norm{z}^2.
\end{align*}
The above two relations give \begin{displaymath}
  \sin \angle(\Range(X_{n_{ev}-1}), u)
  \leq |\tan \angle(\Range(X_{n_{ev}-1}), u)|
  \leq \frac{\norm{\Omega_k z}}{\norm{z}}
  \leq \norm{\Omega_k}.
\end{displaymath}
According to the definition of $\Omega_k$ in \cref{eq: SI-orth: def: Omega_k}, we have
\begin{displaymath}
  \norm{ \Omega_k }
  \leq \frac{ \norm{\widetilde{\Gamma}_\perp^k} \norm{\widetilde{X}_\perp^H V^{(0)}} }{ \sigma_{\min}(\Psi_k) }
  = \frac{ \gamma_{n_{ev}}^k \norm{\widetilde{X}_\perp^H V^{(0)}} }{ \sigma_{\min}(\Gamma_\ell^k (X_\ell^H V^{(0)})) }
  \leq \frac{ \gamma_{n_{ev}}^k \norm{X_{\ell,\perp}^H V^{(0)}} }
  { \sigma_{\min}(\Gamma_\ell^k (X_\ell^H V^{(0)})) },
\end{displaymath}
where the last inequality is due to $\Range( \widetilde{X}_\perp ) \subseteq \Range( X_{\ell,\perp} )$.
Moreover, since $X_\ell^H V^{(0)}$ and $\Gamma_\ell$ are nonsingular, it holds that
\begin{displaymath}
  \sigma_{\min}(\Gamma_\ell^k (X_\ell^H V^{(0)})) \geq \sigma_{\min}(\Gamma_\ell^k) \cdot \sigma_{\min}(X_\ell^H V^{(0)}) = \gamma_\ell^k \cdot \sigma_{\min}(X_\ell^H V^{(0)}).
\end{displaymath}
Combining the above three relations with \cref{eq: SI-orth: norm of X_perp^H V^(k)}, we obtain
\begin{displaymath}
  \norm{\widetilde{X}_\perp^H V^{(k)}}
   \leq \norm{\Omega_k}
  \leq \left( \frac{\gamma_{n_{ev}}}{\gamma_\ell} \right)^k \frac{ \norm{X_{\ell,\perp}^H V^{(0)}} }{ \sigma_{\min}(X_\ell^H V^{(0)}) }
  = \left( \frac{\gamma_{n_{ev}}}{\gamma_\ell} \right)^k \tan \angle(\mathcal{V}^{(0)},\mathcal{X}_\ell),
\end{displaymath}
where the last equality follows from \cref{def: canonical angles between subspaces} and \cref{prop: sin angle between subspaces}.

According to the above relation and \cref{lem: norm bounds of S}, we get
\begin{align*}
  \norm{ \widetilde{X}_\perp^H S^{(k)} }^2
  &\leq \norm{ \widetilde{X}_\perp^H (F_d(p_0)(A) V^{(k)}, \dots, F_d(p_{M-1})(A) V^{(k)}) }^2 \\
  &\leq \sum_{j = 0}^{M-1} \norm{ F_d(p_j)(\widetilde{\Lambda}_\perp) \widetilde{X}_\perp^H V^{(k)} }^2 \\
  &\leq \left( \frac{\gamma_{n_{ev}}}{\gamma_\ell} \right)^{2k} \tan^2 \angle(\mathcal{X}_\ell, \mathcal{V}^{(0)}) \sum_{j = 0}^{M-1} \norm{ F_d(p_j)(\widetilde{\Lambda}_\perp) }^2.
\end{align*}
Combining the above relation with \cref{eq: original upper bound of n_ev-th singular value of S^(k),eq: sum of norms of F_d(p_j)(Lambda_perp)} yields
\begin{equation} \label{eq: SI-orth: final upper bound of n_ev-th singular value of S^(k)}
  \sigma_{n_{ev}}(S^{(k)})
  \leq \left( \frac{\gamma_{n_{ev}}}{\gamma_\ell} \right)^k \sqrt{M} \tan \angle(\mathcal{V}^{(0)},\mathcal{X}_\ell)\rightarrow 0.
\end{equation}

Next we establish a lower bound for $\sigma_1(S^{(k)})$. Decompose
\begin{displaymath}
  V^{(0)} = X_\ell (X_\ell^H V^{(0)}) + X_{\ell,\perp} (X_{\ell,\perp}^H V^{(0)}).
\end{displaymath}
Then similarly to \cref{eq: SI-orth: expression of P^k V} and \cref{eq: SI-orth: norm of X_perp^H V^(k)}, we have
\begin{align}
  &P^k V^{(0)}
  = X_\ell \Gamma_\ell^k (X_\ell^H V^{(0)}) + X_{\ell,\perp} \Gamma_{\ell,\perp}^k (X_{\ell,\perp}^H V^{(0)}), \nonumber\\
  &\norm{X_\ell^H V^{(k)}}
  = \max_{u \in P^k \mathcal{V}^{(0)}, \norm{u} = 1} \norm{X_\ell^H u}
  = \max_{u \in P^k \mathcal{V}^{(0)}, u \neq 0} |\cos \angle(\mathcal{X}_\ell, u)|. \label{Cos}
\end{align}
Write $u = P^k V^{(0)} z$. Then
\begin{align*}
  |\tan \angle(\mathcal{X}_\ell, u)|
  &= \frac{\norm{X_{\ell,\perp}^H u}}{\norm{X_\ell^H u}}
  = \frac{\norm{\Gamma_{\ell,\perp}^k (X_{\ell,\perp}^H V^{(0)}) z}}{\norm{\Gamma_\ell^k (X_\ell^H V^{(0)}) z}}
  \leq \frac{\norm{\Gamma_{\ell,\perp}^k (X_{\ell,\perp}^H V^{(0)})}}{\sigma_{\min}(\Gamma_\ell^k (X_\ell^H V^{(0)}))} \\
  &\leq \frac{\norm{\Gamma_{\ell,\perp}^k}}{\sigma_{\min}(\Gamma_\ell^k)}
        \cdot \frac{\norm{X_{\ell,\perp}^H V^{(0)}}}{\sigma_{\min}(X_\ell^H V^{(0)})}
  = \left( \frac{\gamma_{\ell+1}}{\gamma_\ell} \right)^k \tan \angle(\mathcal{V}^{(0)},\mathcal{X}_\ell)\\
  &\leq \tan \angle(\mathcal{V}^{(0)},\mathcal{X}_\ell), 
\end{align*}
from which and \eqref{Cos} it follows that
\begin{equation}\label{lowboundk}
  \norm{X_\ell^H V^{(k)}} \geq \cos \angle(\mathcal{V}^{(0)},\mathcal{X}_\ell).
\end{equation}
Therefore, by \cref{lem: norm bounds of S},
$\norm{S^{(k)}} \geq \norm{X_\ell^H S^{(k)}}$, $ X_\ell^H P V^{(k)}=\Gamma_\ell X_\ell^H
V^{(k)}$ and \eqref{lowboundk}, we obtain
\begin{equation} \label{eq: SI-orth: final lower bound of largest singular value of S^(k)}
  \begin{aligned}
    \sigma_1(S^{(k)})
    &\geq \norm{X_\ell^H S^{(k)}}
    = \norm{ X_\ell^H (F_d(p_0)(A) V^{(k)}, \dots, F_d(p_{M-1})(A) V^{(k)}) } \\
    &\geq \norm{ \Gamma_\ell X_\ell^H V^{(k)} }
    \geq \gamma_\ell \cdot \norm{X_\ell^H V^{(k)}}
    \geq \gamma_\ell \cdot \cos \angle(\mathcal{V}^{(0)},\mathcal{X}_\ell).
  \end{aligned}
\end{equation}

Finally, bound \cref{eq: ratio of singular values of S^(k) for SI-orth} follows from \cref{eq: SI-orth: final upper bound of n_ev-th singular value of S^(k),eq: SI-orth: final lower bound of largest singular value of S^(k)}.
\end{proof}

\begin{remark}
The preceding two theorems have proved \eqref{infbound}. Under the assumptions of \cref{thm: ratio of singular values of S^(k) for SI-orth}, the lower bound in \eqref{eq: ratio of singular values of S^(k) for SI-orth} tends to infinity as $k$ increases, so does
$\frac{\sigma_1(S^{(k)})}{\sigma_{n_{ev}}(S^{(k)})}$.
Moreover,
since $\gamma_1 \geq \gamma_\ell > \gamma_{n_{ev}}$,
it is expected that
$\frac{\sigma_1(S^{(k)})}{\sigma_{n_{ev}}(S^{(k)})}$ in \eqref{eq: ratio of singular values of S^(k) for SI-orth} increases more slowly
than that in \eqref{eq: ratio of singular values of S^(k) for SI-like}
when $ \gamma_1 \neq \gamma_{\ell}$.
\end{remark}


\subsection{The combination restart}\label{section:3.2}

The results in the last subsection lead to the following consequences:
In finite precision arithmetic,
for a given dropping threshold $100\epsilon_{\rm mach}\|S^{(k)}\|_F$ with $\epsilon_{\rm mach}$
being machine precision and $\|\cdot\|_F$ being the Frobenius norm of a matrix,
the numerical rank of $S^{(k)}$ is ultimately smaller than $n_{ev}$ as $k$ increases. If we use
the TSVD or rank revealing QR decomposition of $S^{(k)}$ to obtain an effective subspace with
the dimension smaller than $n_{ev}$,
one definitely fails to find all the $n_{ev}$ desired eigenpairs.
On the other hand,
if one simply computes the QR decomposition of $S^{(k)}$, then, for sufficiently large $k$,
the {\em computed} subspace generated by
the Q-factor bears no relation to the true $\mathcal{R}(S^{(k)})$
whenever $S^{(k)}$ starts to be numerically rank deficient,
causing that the restarted algorithm may fail to obtain sufficiently accurate approximations to
some columns of
$X_{n_{ev}}$ and the residual norms of some refined Ritz pairs with the refined Ritz values in $[a,b]$
start to increase; that is, for a prescribed stopping tolerance $O(\epsilon_{\rm mach})$, the restarted
algorithm first converges but afterwards starts to diverge before it reaches $O(\epsilon_{\rm mach})$.

To fix the above intrinsic deficiencies, a fundamental cure
is to construct a well-conditioned
$S^{(k+1)}$ at restart $k+1$ that contains the information on $X_{n_{ev}}$ as much as possible.
This motivates us to seek a new
restart scheme that is both numerically
stable and makes full use of the best possible information on $X_{n_{ev}}$ that is
available at restart $k$.
A natural approach is to retain the best available approximations in $\Range(S^{(k)})$ to
the $n_{ev}$ desired
eigenvectors $x_1,\ldots,x_{n_{ev}}$ in $\Range(S^{(k+1)})$. Borrowing the restart scheme in \cite{jia_block} for block Krylov subspace methods solving eigenvalue problems,
unlike the previous two $V^{(k+1)}$ in \eqref{eq: def: SI restart} and \eqref{eq: def: SI-orth restart},
we take
\begin{equation} \label{eq: def: combination restart}
  V^{(k+1)} = \widetilde{X}^{(k)} C \in  \mathbb{C}^{n\times\ell}
\end{equation}
at restart $k+1$,
where the columns of $\widetilde{X}^{(k)}\in \mathbb{C}^{n\times n_{ev}}$ are
the currently available approximations to $X_{n_{ev}}$,
and $C\in \mathbb{R}^{n_{ev}\times \ell}$ is generated randomly from the
standard normal distribution, i.e., a Gaussian random matrix.
Since each column of $ V^{(k+1)}$ is a random linear combination of those of $\widetilde{X}^{(k)}$,
we call $V^{(k+1)}$ in \eqref{eq: def: combination restart} a \emph{combination restart}.
Obviously, one can take $\widetilde{X}^{(k)}$ to be those selected Ritz vectors in
the CJ--SS--RR method and refined Ritz vectors in the CJ--SS--RRR method.
Importantly, this restart scheme straightforwardly applies to the SS--RR and SS--RRR methods
for general standard and generalized eigenvalue problems.

In the above descriptions, we assume that the number $n_{ev}$ of desired eigenpairs is known
in advance. In practice, one has to estimate $n_{ev}$ reliably. For large problems, stochastic techniques
have been developed in \cite{CJ_FEAST_cross,CJ_FEAST_augmented}
that use the CJ series expansion to estimate $n_{ev}$ efficiently and reliably.


We notice that, when $\ell \geq n_{ev}$,
the assumptions in \cref{thm: ratio of singular values of S^(k) for SI-orth} do not hold any longer. 
In this case, we simply take the moment number $M=1$,
and the corresponding CJ--SS--RR method with the natural restart \cref{eq: def: SI-orth restart} degenerates to the numerically stable CJ-FEAST eigensolvers \cite{CJ_FEAST_cross,CJ_FEAST_augmented} for
Hermitian eigenvalue problems.
In what follows, we exclude the situation of $\ell \geq n_{ev}$, and
always assume that $\ell < n_{ev}$ and $C$ is of full column rank.

To have in-depth understanding for the combination restart \cref{eq: def: combination restart},
let us first consider the ideal case where $\widetilde{X}^{(k-1)} = X_{n_{ev}}$.
We have
\begin{equation} \label{eq: thick-like: eigenvectors as restart vectors}
  \begin{aligned}
    S^{(k)}
    &= \left( F_d(p_0)(A) X_{n_{ev}} C, \dots, F_d(p_{M-1})(A) X_{n_{ev}} C \right) \\
    &= X_{n_{ev}} \left( F_d(p_0)(\Lambda_{n_{ev}}) C, \dots, F_d(p_{M-1})(\Lambda_{n_{ev}}) C \right)
    =: X_{n_{ev}} T,
    \end{aligned}
\end{equation}
i.e., $\Range(S^{(k)})=\Range(X_{n_{ev}})$ when $T\in \mathbb{R}^{n_{ev}\times M\ell}$
is of full row rank, and
\begin{displaymath}
  \frac{\sigma_1(S^{(k)})}{\sigma_{n_{ev}}(S^{(k)})}
  = \frac{\sigma_1(T)}{\sigma_{n_{ev}}(T)}
  = \kappa(T).
\end{displaymath}
Next we establish a compact bound for $\kappa(T)$.

\begin{theorem}\label{idealkappa}
For $\widetilde{X}^{(k-1)} = X_{n_{ev}}$ and
the matrix $T$ defined in \eqref{eq: thick-like: eigenvectors as restart vectors},
let $C^T=U_c(\Sigma_c\ \ 0)(V_c,V_{c,\perp})^T$ be the full SVD of $C^T$ and $z$ be the left singular vector of $T$ corresponding to $\sigma_{\min}(T)$. Then
\begin{equation}\label{idealratio}
\frac{\sigma_1(S^{(k)})}{\sigma_{n_{ev}}(S^{(k)})}= \kappa(T)\leq \frac{\sqrt{M}\kappa(C)}{\|V_c^TF_d(p_0)(\Lambda_{n_{ev}})z\|}.
\end{equation}
\end{theorem}

\begin{proof}
From the definition of $T$, since $p_j$ and $\Lambda_{n_{ev}}$ are real, we obtain
\begin{equation}\label{sumT}
TT^T=\sum_{j=0}^{M-1}F_d(p_j)(\Lambda_{n_{ev}}) C C^TF_d(p_j)(\Lambda_{n_{ev}})\in \mathbb{R}^{n_{ev}\times
n_{ev}}.
\end{equation}
By \cref{thm: bounds of F_d(p)(t)}, we have $\|F_d(p_j)(\Lambda_{n_{ev}})\|\leq 1$.
Therefore,
\begin{equation}\label{upbound}
\sigma_{\max}^2(T)=\sigma_{\max}(TT^T) \leq \sum_{j=0}^{M-1}\|F_d(p_j)(\Lambda_{n_{ev}}) C\|^2
\leq M\|C\|^2.
\end{equation}
Notice that
$$
\sigma_{\min}^2(T)=\sigma_{\min}(TT^T)=z^TTT^Tz
$$
and each term in the right-hand side of \eqref{sumT} is symmetric positive semi-definite. Then
exploiting the SVD of $C^T$, we obtain
\begin{eqnarray}
\sigma_{\min}^2(T) &\geq & z^TF_d(p_0)(\Lambda_{n_{ev}}) C C^TF_d(p_0)(\Lambda_{n_{ev}})z \nonumber\\
&=&\|C^TF_d(p_0)(\Lambda_{n_{ev}})z\|^2 \nonumber\\
&\geq &\sigma_{\min}^2(C) \|V_c^TF_d(p_0)(\Lambda_{n_{ev}})z\|^2.  \label{lowbound}
\end{eqnarray}
Bound \eqref{idealratio} then follows from \eqref{upbound}, \eqref{lowbound} and \eqref{eq: eigenvalues of F_d(1)(A)}.
\end{proof}

\begin{remark}\label{rem3.10}
$\|V_c^TF_d(p_0)(\Lambda_{n_{ev}})z\|$ is the projection size of $F_d(p_0)(\Lambda_{n_{ev}})z$
onto the column space of $C$. Notice that the smallest eigenvalue (singular value)
of the symmetric positive semi-definite matrix $F_d(p_0)(\Lambda_{n_{ev}})$ is $\gamma_{n_{ev}}$ and that
$$
\|V_c^TF_d(p_0)(\Lambda_{n_{ev}})z\|^2+\|V_{c,\perp}^TF_d(p_0)(\Lambda_{n_{ev}})z\|^2
=\|F_d(p_0)(\Lambda_{n_{ev}})z\|^2.
$$
Then
for a Gaussian random matrix $C\in \mathbb{R}^{n_{ev}\times \ell}$
and $V_c\in \mathbb{R}^{n_{ev}\times \ell}$, since
$V_{c,\perp}\in \mathbb{R}^{n_{ev}\times (n_{ev}-\ell)}$, we approximately
have
$$
\|V_c^TF_d(p_0)(\Lambda_{n_{ev}})z\|
\approx \sqrt{\frac{\ell}{n_{ev}}}\|F_d(p_0)(\Lambda_{n_{ev}})z\|\geq
\sqrt{\frac{\ell}{n_{ev}}}\gamma_{n_{ev}}.
$$
Therefore, by $n_{ev}\leq M\ell$ and $\gamma_{n_{ev}}>\frac{1}{4}$ under the assumption \eqref{eq: theoretical requirement on degree}, bound~\eqref{idealratio} is approximately no more than
$$
4\sqrt{M}\sqrt{\frac{n_{ev}}{\ell}}\kappa(C)\leq 4M\kappa(C).
$$
On the other hand, for the Gaussian random matrix $C$, it is well known from, e.g.,
\cite[Theorem 8.4]{Martinsson_Tropp_2020} that
\begin{equation}\label{kappaT}
\kappa(C)\leq \frac{\sqrt{n_{ev}}+\sqrt{\ell}}{\sqrt{n_{ev}}-\sqrt{\ell}}
\end{equation}
with high probability,
indicating that $\kappa(C)$ is generally generic and the upper bound in \eqref{idealratio} is of modest size
provided that $\ell$ is considerably smaller than $n_{ev}$.

\end{remark}

\begin{remark} \label{rem3.11}
For $\widetilde{X}^{(k-1)}\not=X_{n_{ev}}$, by continuity,
$\sigma_1(S^{(k)}) / \sigma_{n_{ev}}(S^{(k)})$ approaches $\kappa(T)$ as $k$ increases. Consequently,
the numerical rank of $S^{(k)}$ is no smaller than $n_{ev}$, meaning that
the combination restart is numerically stable and the restarted subspaces $\Range(S^{(k)})$ retain
effective information on $X_{n_{ev}}$. Generally, it is expected that
the ratio $\sigma_1(S^{(k)}) / \sigma_{n_{ev}}(S^{(k)})$ is of modest size
even though $\widetilde{X}^{(k-1)}$ is not close to $X_{n_{ev}}$.
\end{remark}

We next consider a general $\widetilde{X}^{(k-1)} \in \Complex^{n \times n_{ev}}$, and establish the following result.

\begin{theorem} \label{thm: convergence of combination restart}
  Given $\widetilde{X}^{(k-1)} \in \Complex^{n \times n_{ev}}$
  and $C \in \Real^{n_{ev} \times \ell}$,
  let $\delta$ and $V^{(k)}$ be defined as in \cref{eq: def: Delta_min and delta,eq: def: combination restart}, respectively,
  write $\widetilde{\mathcal{X}}^{(k-1)} = \Range(\widetilde{X}^{(k-1)})$, $\mathcal{X}_{n_{ev}} = \Range(X_{n_{ev}})$, $\mathcal{S}^{(k)} = \Range(S^{(k)})$,
  and
\begin{equation} \label{eq: def: T^(k)}
    T^{(k)} = \left(
      F_d(p_0)(\Lambda_{n_{ev}})X_{n_{ev}}^H V^{(k)}, \dots,
      F_d(p_{M-1})(\Lambda_{n_{ev}})X_{n_{ev}}^H V^{(k)}
    \right) \in \Complex^{n_{ev} \times M \ell},
  \end{equation}
  and assume that $X_{n_{ev}}^H \widetilde{X}^{(k-1)}$ is nonsingular,
  $T^{(k)}$ is of full row rank,
  the CJ series degree $d$ satisfies \cref{eq: theoretical requirement on degree},
  and
  \begin{displaymath}
    1 \leq \ell <n_{ev} \leq M \ell < n.
  \end{displaymath}
  Then it holds that
  \begin{equation} \label{eq: convergence of combination restart}
    \tan \angle \left( \mathcal{S}^{(k)}, \mathcal{X}_{n_{ev}} \right)
    \leq \frac{\delta \sqrt{M}}{1/2 - \delta}
    \cdot\kappa(T^{(k)}) \cdot
     \tan \angle(\widetilde{\mathcal{X}}^{(k-1)}, \mathcal{X}_{n_{ev}})
  \end{equation}
  with $\delta<\frac{1}{4}$.
\end{theorem}

\begin{proof}
Denote $X_\perp = X_{n_{ev},\perp}$, $\Lambda_\perp = \Lambda_{n_{ev},\perp}$,
and express
\begin{displaymath}
  \widetilde{X}^{(k-1)} = X_{n_{ev}}\Psi + X_\perp\Phi
\end{displaymath}
as the orthogonal sum,
where $\Psi = X_{n_{ev}}^H\widetilde{X}^{(k-1)}$ and $\Phi = X_\perp^H\widetilde{X}^{(k-1)}$.
Since $\Psi$ is nonsingular by assumption, it is from
\cite[Theorem~3.1]{tangent_angle} that
\begin{equation} \label{eq: tan angle between subspaces}
  \tan\angle(\widetilde{\mathcal{X}}^{(k-1)},\mathcal{X}_{n_{ev}})
  = \norm{X_\perp^H \widetilde{X}^{(k-1)} (X_{n_{ev}}^H \widetilde{X}^{(k-1)})^{-1}}
  = \norm{\Phi\Psi^{-1}}.
\end{equation}
From $V^{(k)} = \widetilde{X}^{(k-1)}C$, we have $X_{n_{ev}}^H V^{(k)} = \Psi C$ and $X_\perp^H V^{(k)} = \Phi C$.
Hence
\begin{equation} \label{eq: rewrite S^{(k)}}
\begin{aligned}
  S^{(k)}
  &= X_{n_{ev}}
    \bigl(F_d(p_0)(\Lambda_{n_{ev}})\Psi C,\dots,
    F_d(p_{M-1})(\Lambda_{n_{ev}})\Psi C\bigr) \\
  &\quad + X_\perp
    \bigl(F_d(p_0)(\Lambda_\perp)\Phi C,\dots,
    F_d(p_{M-1})(\Lambda_\perp)\Phi C\bigr) \\
  &=: X_{n_{ev}}\widetilde{\Psi} + X_\perp\widetilde{\Phi},
\end{aligned}
\end{equation}
where $\widetilde{\Psi} = T^{(k)}$ is of row full rank.

Take an arbitrary unit vector $x = X_{n_{ev}}y\in\mathcal{X}_{n_{ev}}$ with $\|y\|=1$.
Then 
we have
\begin{displaymath}
  \tan\angle(\mathcal{S}^{(k)}, x)
  = \min_{s\in\mathcal{S}^{(k)}\setminus \{0\}} \tan\angle(s,x).
\end{displaymath}
Since $\widetilde{\Psi}$ is of full row rank, we have $\widetilde{\Psi}\widetilde{\Psi}^\dagger = I_{n_{ev}}$,
where $\widetilde{\Psi}^\dagger$ is the Moore--Penrose pseudoinverse of $\widetilde{\Psi}$. As a result,
by \eqref{eq: rewrite S^{(k)}}, the vector
\begin{displaymath}
  s = S^{(k)} \widetilde{\Psi}^\dagger y
  = X_{n_{ev}}y + X_\perp \widetilde{\Phi}\widetilde{\Psi}^\dagger y
  = x + X_\perp \widetilde{\Phi}\widetilde{\Psi}^\dagger y.
\end{displaymath}
Then
\begin{displaymath}
  \tan\angle(\mathcal{S}^{(k)}, x)
  \leq \tan\angle(s,x)
  = \norm{X_\perp \widetilde{\Phi}\widetilde{\Psi}^\dagger y}
  \leq \norm{\widetilde{\Phi}\widetilde{\Psi}^\dagger}.
\end{displaymath}

Let the columns of $S_{k,\perp}$
be an orthonormal basis of the orthogonal complement of $\mathcal{S}^{(k)}$
with respect to $\mathbb{C}^n$.
Then by \cref{prop: sin angle between subspaces},
since $\Rank(S^{(k)}) \geq \Rank(\widetilde{\Psi}) = n_{ev}$,
we obtain
\begin{displaymath}
  \sin \angle(\mathcal{S}^{(k)},\mathcal{X}_{n_{ev}})
  = \norm{S_{k,\perp}^H X_{n_{ev}}}
  = \max_{x\in\mathcal{X}_{n_{ev}}, \norm{x}=1} \norm{S_{k,\perp}^H x}
  = \max_{x\in\mathcal{X}_{n_{ev}}\setminus \{0\}} \sin\angle(\mathcal{S}^{(k)},x),
\end{displaymath}
and consequently
\begin{equation} \label{eq: tan-S-X}
  \tan\angle(\mathcal{S}^{(k)},\mathcal{X}_{n_{ev}})
  = \max_{x\in\mathcal{X}_{n_{ev}}\setminus \{0\}} \tan\angle(\mathcal{S}^{(k)},x)
  \leq \norm{\widetilde{\Phi}\widetilde{\Psi}^\dagger}
  \leq \kappa(\widetilde{\Psi})\frac{\norm{\widetilde{\Phi}}}{\norm{\widetilde{\Psi}}}.
\end{equation}

We now bound $\norm{\widetilde{\Phi}}$ and $\norm{\widetilde{\Psi}}$.
First,
\cref{eq: rewrite S^{(k)}} shows that
\begin{equation}\label{Psitilde}
  \widetilde{\Phi} = \bigl(F_d(p_0)(\Lambda_\perp),\dots,F_d(p_{M-1})(\Lambda_\perp)\bigr)\,(I_M\otimes (\Phi C)).
\end{equation}
Notice that 
\begin{displaymath}
\begin{aligned}
  \normbig{\bigl(F_d(p_0)(\Lambda_\perp),\dots,F_d(p_{M-1})(\Lambda_\perp)\bigr)}^2
  &= \lambda_{\max}\left( \sum_{k=0}^{M-1}\bigl(F_d(p_k)(\Lambda_\perp)\bigr)^2 \right) \\
  &= \max_{\lambda\in\lambda(A)\setminus[a,b]} \sum_{k=0}^{M-1}\abs{F_d(p_k)(\lambda)}^2.
\end{aligned}
\end{displaymath}
By \cref{eq: error bound of F_d(p)(t)}, the above right-hand side is at most $M \delta^2$.
Moreover, $\norm{I_M\otimes(\Phi C)} = \norm{\Phi C}$,
and
\begin{displaymath}
  \norm{\Phi C}
  = \norm{\Phi\Psi^{-1}\Psi C}
  \leq \norm{\Phi\Psi^{-1}} \cdot \norm{\Psi C}.
\end{displaymath}
Thus, from \eqref{Psitilde}, we have
\begin{equation} \label{eq:bound-Phi}
  \norm{\widetilde{\Phi}}
  \leq \sqrt{M}\delta \cdot \norm{\Phi\Psi^{-1}} \cdot \norm{\Psi C}.
\end{equation}

Second,
according to \cref{eq: rewrite S^{(k)},lem: norm bounds of S}, we obtain
\begin{displaymath}
  \norm{\widetilde{\Psi}}
  \geq \norm{F_d(p_0)(\Lambda_{n_{ev}})\Psi C}
  \geq \sigma_{\min}\left( F_d(p_0)(\Lambda_{n_{ev}}) \right) \cdot \norm{\Psi C}.
\end{displaymath}
By \cref{eq: def: Delta_min and delta},
the assumption that $d$ satisfies \cref{eq: theoretical requirement on degree} means
that $\delta < \frac14$.
It follows from \cref{eq: error bound of F_d(p)(t)} that
\begin{displaymath}
  \sigma_{\min}\left( F_d(p_0)(\Lambda_{n_{ev}}) \right)
  = \min_{\lambda\in\lambda(A)\cap[a,b]}\abs{F_d(p_0)(\lambda)}
  \geq \frac12-\delta >\frac{1}{4}.
\end{displaymath}
Therefore,
\begin{equation} \label{eq:bound-Psi}
  \norm{\widetilde{\Psi}}
  \geq (\frac12-\delta) \cdot \norm{\Psi C}.
\end{equation}
Substituting \cref{eq:bound-Phi,eq:bound-Psi} into \cref{eq: tan-S-X},
we obtain
\begin{displaymath}
  \tan \angle \left( \mathcal{S}^{(k)}, \mathcal{X}_{n_{ev}} \right)
  \leq  \frac{\delta \sqrt{M}}{1/2 - \delta} \cdot
    \kappa(\widetilde{\Psi}) \cdot
    \norm{\Phi\Psi^{-1}}.
\end{displaymath}
Finally,
using \cref{eq: tan angle between subspaces} and $\widetilde{\Psi} = T^{(k)}$ completes the proof.
\end{proof}

\begin{remark}
  In accordance with the definitions of $T$ and $T^{(k)}$
  in \cref{eq: thick-like: eigenvectors as restart vectors,eq: def: T^(k)},
  it follows that $T^{(k)} \to T$ as $\widetilde{X}^{(k-1)} \to X_{n_{ev}}$.
  Consequently, for $k$ sufficiently large,
  relation \cref{eq: convergence of combination restart} is asymptotically
  \begin{equation} \label{eq: convergence of combination restart with T}
    \tan \angle \left( \mathcal{S}^{(k)}, \mathcal{X}_{n_{ev}} \right)
    \lesssim \frac{\delta \sqrt{M}}{1/2 - \delta}  \cdot \kappa(T)
    \cdot \tan \angle(\widetilde{\mathcal{X}}^{(k-1)}, \mathcal{X}_{n_{ev}}).
  \end{equation}
\end{remark}
Notice that under the condition \cref{eq: theoretical requirement on degree}, we have $\frac{\delta}
{1/2 - \delta}<1$.
\cref{thm: convergence of combination restart} indicates that, for $\delta$ sufficiently small, once
$\widetilde{X}^{(k-1)}$ extracted from $\mathcal{S}^{(k-1)}$
is reasonably close to $X_{n_{ev}}$,
the combination restart \cref{eq: def: combination restart} guarantees that
the restarted subspace $\mathcal{S}^{(k)}$ generally
contains much better approximations to the desired eigenspace
$\mathcal{X}_{n_{ev}}$ than $\mathcal{S}^{(k-1)}$ does, meaning that the new
$\widetilde{X}^{(k)}$ extracted from $\mathcal{S}^{(k)}$ is generally much more accurate than $\widetilde{X}^{(k-1)}$.

\begin{remark}
Combining the requirement $\frac{\delta \sqrt{M}}{1/2 - \delta}\kappa(T)<1$ with \cref{eq: def: Delta_min and delta}, we see that the polynomial degree $d$ satisfies
\begin{displaymath}
  d > \frac{\pi^2 ( 1 + \sqrt{M}\kappa(T) )^{1/3} }{ \Delta_{\min}^{4/3} } - 2
  \approx \frac{\sqrt[3]{\kappa(T)} \pi^2 M^{1/6}}{ \Delta_{\min}^{4/3} } - 2,
\end{displaymath}
in which we can replace $\kappa(T)$ by its bound in \eqref{kappaT}.
In computations,
we use an empirical choice of the CJ series degree $d$:
\begin{equation} \label{eq: degree estimate for practice}
  d = \left\lceil D \cdot \frac{\sqrt[3]{2} \pi^2 M^{1/6}}{(b-a)^{4/3}} \right\rceil - 2,
\end{equation}
where $D \in [1, 8]$, which is similar to that suggested in \cite{CJ_FEAST_cross} when $M=1$.
\end{remark}

\section{A CJ--SS--RRR algorithm with the hybrid restart}
\label{sec: CJ--SS--RRR with hybrid restart}

\subsection{The selection of restart vectors}

The preceding theorems and remarks demonstrate that, under the assumptions that
$n_{ev}$ is known and $\widetilde{X}^{(k)}$ is selected correctly,
the numerical rank of $S^{(k)}$ generated by the combination restart is no smaller than
$n_{ev}$ and $\widetilde{X}^{(k)} \rightarrow X_{n_{ev}}$ as $k$ increases.
However, there are possibly spurious Ritz pairs and refined Ritz pairs that must be identified.
The authors \cite{CJ_SS_RRR} have shown that the Rayleigh--Ritz procedure itself cannot
identify and remove spurious Ritz pairs reliably but the refined Rayleigh--Ritz procedure \cref{eq: def: refined vector} that computes the refined Ritz vectors
can be embedded with a reliable removal of spurious refined Ritz values and extra refined
Ritz vectors, so that it works more reliably
and efficiently than the Rayleigh--Ritz projection.
As a result, we propose to use those refined Ritz vectors obtained by the refined removal approach
as $\widetilde{X}^{(k)}$ in the combination restart \cref{eq: def: combination restart}, which leads to
a restarted CJ--SS--RRR algorithm. The extensive experiments in \cite{CJ_SS_RRR} have illustrated
that the number of the refined Ritz vectors after removal is exactly $n_{ev}$, provided that their relative
residual norms drop below $10^{-3}$ and even $10^{-2}$.

Similarly, when taking the Ritz vectors to form $\widetilde{X}^{(k)}$, we get the CJ--SS--RR algorithm with the combination restart, but those
removal approaches used nowadays are not reliable and lack a theoretical support \cite{CJ_SS_RRR}.

\subsection{The hybrid restart scheme}
\label{subsec: hybrid restart}

For the combination restart \cref{eq: def: combination restart}, \cref{thm: convergence of combination restart} shows that
the closer the columns of $\widetilde{X}^{(k-1)}$ are to the $n_{ev}$ desired eigenvectors,
the more accurate the restarted subspace $\Range(S^{(k)})$ is, leading to
faster convergence.
Conversely, if $\widetilde{X}^{(k-1)}$ is not close to the desired eigenvectors,
$\Range(S^{(k)})$ is not so good, causing that the algorithm converges slowly.
Therefore, when employing the combination restart,
it is desirable that $\widetilde{X}^{(k-1)}$ has already had some accuracy.

The above considerations motivate us to consider a hybrid restart:
(i) apply the natural restart in the initial phase
until the approximate eigenpairs have some accuracy;
(ii) switch to the combination restart in order to ensure numerical stability and
accelerate the convergence of
the corresponding restarted algorithms.



The advantage of the natural restart lies in the property \cref{eq: SI property of S^(k)} of subspace iteration, which successively improves $\Range(S^{(k)})$ regularly
and guarantees that the restarted
algorithms compute more accurate approximate eigenpairs as $k$ increases,
though the convergence may be slow,
provided that the size of $\sigma_1(S^{(k)}) / \sigma_{n_{ev}}(S^{(k)})$ is modest.
The advantage of the combination restart is the higher effectiveness of $\Range(S^{(k)})$ and
its numerical stability once $\widetilde{X}^{(k-1)}$ starts to have some accuracy,
which makes the restarted algorithms converge faster.
Moreover, provided that the CJ series degree $d$ satisfies \cref{eq: theoretical requirement on degree} and $1 \leq \ell< n_{ev}$, from \eqref{eq: eigenvalues of F_d(1)(A)} we have
\begin{displaymath}
  1
  \leq \frac{\gamma_\ell}{\gamma_{n_{ev}}}
  \leq \frac{\gamma_1}{\gamma_{n_{ev}}}
  \leq 4.
\end{displaymath}
Therefore, bounds \eqref{eq: ratio of singular values of S^(k) for SI-like} and \eqref{eq: ratio of singular values of S^(k) for SI-orth} indicate that
the application of the natural restart in the initial stage does not
cause numerical instability for small $k$.
Consequently,
the hybrid restart scheme exploits the advantages of both the natural restart and the combination restart
but avoids the drawbacks of each of them,
thereby enhancing the performance of restarted algorithms.

In computations, we use the relative residual norms $\norm{\hat{r}_i^{(k)}}$ of the refined Ritz vectors
to design a switching criterion between the two kinds of restarts.
We first employ the natural restart and then switch to the combination restart
when
\begin{equation} \label{eq: switch criterion}
  \max_{1 \leq i \leq n_{ev}} \norm{\hat{r}_i^{(k)}} \leq \norm{A} \cdot \epsilon_{\rm switch}
\end{equation}
for a given small threshold $\epsilon_{\rm switch}$. 
To be robust, we add a guard switching criterion:
If the current maximum residual norm is greater than the one in the previous restart:
\begin{equation} \label{eq: guard switch criterion}
  \max_{1 \leq i \leq n_{ev}} \norm{\hat{r}_i^{(k)}} > \max_{1 \leq i \leq n_{ev}} \norm{\hat{r}_i^{(k-1)}}, \end{equation}
we also switch to the combination restart. The occurrence of \eqref{eq: guard switch criterion}
is due to numerical instability, i.e., a very large $\sigma_1(S^{(k)}) / \sigma_{n_{ev}}(S^{(k)})$.
This guard criterion ensures that the algorithm switches to the combination restart
once the natural restarts suffer from numerical instability. We have numerically found that $\epsilon_{\rm switch}\approx \sqrt{tol}$,
where $tol$ is the stopping tolerance in \eqref{eq: convergence criterion}, is a good choice and that
\eqref{eq: guard switch criterion} occurs only when $tol\leq 10^{-12}$.


\subsection{The complete algorithm}

Having discussed the combination restart and the hybrid restart,
we now describe a complete CJ--SS--RRR algorithm with the hybrid restart in \cref{alg: CJSSRRR_hybrid}.
Since \cref{thm: ratio of singular values of S^(k) for SI-like,thm: ratio of singular values of S^(k) for SI-orth} indicate that the natural restart \cref{eq: def: SI-orth restart} has better numerical stability than the naive natural restart \cref{eq: def: SI restart},
we use \cref{eq: def: SI-orth restart} as the initial restart in \cref{alg: CJSSRRR_hybrid}.
In step~\ref{step: CJSSRRR_hybrid: convergence test},
the same convergence criterion \cref{eq: convergence criterion} as in \cref{alg: CJSSRRR} is used to test whether the refined Ritz pairs have converged.

\begin{algorithm}[htbp]
\caption{The CJ--SS--RRR algorithm with the hybrid restart}
\label{alg: CJSSRRR_hybrid}
\begin{algorithmic}[1]
  \REQUIRE The matrix $A$, the interval $[a,b]$, the series degree $d$, the moment number $M$, and
  a column orthonormal matrix $V^{(0)}\in \mathbb{C}^{n\times \ell}$ with $M\ell \geq n_{ev}$.
 \STATE Initially take the natural restart \cref{eq: def: SI-orth restart}.
  \STATE Compute the CJ coefficients by \cref{eq: def: Chebyshev coefficient,eq: def: Jackson damping factor}.
  \FOR {$k = 0, 1, 2, \dots$}
    \STATE Compute $S^{(k)}$ in \cref{eq: def: k-th search subspace S^(k)} and its QR decomposition
    $S^{(k)} = U^{(k)}R^{(k)}$.
    \STATE Use the Rayleigh--Ritz procedure \cref{eq: def: Rayleigh--Ritz} to get the Ritz values $\tilde{\lambda}_i^{(k)}$. 
    \STATE Compute the refined Ritz pairs $(\hat{\lambda}_i^{(k)}, \hat{x}_i^{(k)})$ by \cref{eq: def: refined vector}, and select them by \cref{alg: removal}.
    \STATE Test the convergence by \cref{eq: convergence criterion}. \label{step: CJSSRRR_hybrid: convergence test}
    \STATE If not converged, restart with the natural restart \cref{eq: def: SI-orth restart} or
    the combination restart \cref{eq: def: combination restart} when either \cref{eq: switch criterion} or \cref{eq: guard switch criterion} is satisfied.
  \ENDFOR
  \RETURN The $n_{ev}$ converged refined Ritz pairs $(\hat{\lambda}_i^{(k)}, \hat{x}_i^{(k)})$.
\end{algorithmic}
\end{algorithm}

Similarly, we have developed a CJ--SS--RR algorithm with the hybrid restart, but our extensive numerical
experiments have demonstrated that the CJ--SS--RR algorithms with the combination and hybrid restarts
were less efficient and reliable than the corresponding CJ--SS--RRR algorithms. Therefore,
we will not pay attention to and report concrete experiments on CJ--SS--RR.

\section{Numerical experiments}
\label{sec: numerical experiments}

The test matrices are from the SuiteSparse Matrix Collection \cite{matrix_collection},
and some of their basic properties and the intervals of interest are listed in \cref{tab: test matrices}.
For each matrix,
interval I is close to one end of the spectrum, and interval II is inside the spectrum.
The minimum and maximum eigenvalues $\lambda_{\min}, \lambda_{\max}$ of each matrix are computed by the
 {\sc Matlab} function \texttt{eigs}
and the number $n_{ev}$ of desired eigenvalues in each interval is computed by the CJ-FEAST eigensolver,
which is a direct adaptation of the FEAST SVDsolvers in \cite{CJ_FEAST_cross,CJ_FEAST_augmented}.
The unit-length desired eigenvectors $x_1, \ldots, x_{n_{ev}}$ are also computed by the CJ-FEAST eigensolver,
and their relative residual norms are all smaller than $10^{-14}$,
the level of machine precision $\epsilon_{\rm mach} = 2.22\times 10^{-16}$.
All the numerical experiments were performed on an Intel Core i7-9700, CPU 3.0GHz, 8GB RAM using {\sc Matlab} R2024b under the Microsoft Windows 10 64-bit system.

\begin{table}[ht]
\caption{Properties of the test matrices and the intervals of interest.}
\label{tab: test matrices}
\centering

\begin{tabular}{*{8}{c}}
\toprule
Matrix            & Size  & $\lambda_{\min}$ & $\lambda_{\max}$ & Interval I  & $n_{ev}$ & Interval II  & $n_{ev}$ \\
\midrule
SiH4              & 5041  & -0.996           & 36.8             & [0.5,2.5]   & 120      & [17.5,18]   & 123      \\
SiNa              & 5743  & -0.705           & 25.7             & [0,2]       & 178      & [12.5,13]   & 183      \\
delaunay\_n13     & 8192  & -3.57            & 6.46             & [5.2,6.2]   & 310      & [1.5,2]     & 312      \\
benzene           & 8219  & -0.730           & 58.4             & [1.5,4.5]   & 229      & [28,29]     & 231      \\
stokes64          & 12546 & -1.70e-3         & 4.00             & [3.8,3.9]   & 232      & [1.9,2.1]   & 232      \\
Pres\_Poisson     & 14822 &  1.28e-5         & 26.1             & [8,13]      & 68       & [4.2,5]     & 66       \\
Si10H16           & 17077 & -1.15            & 37.0             & [0.5,2]     & 307      & [18,18.4]   & 303      \\
rgg\_n\_2\_15\_s0 & 32768 & -5.12            & 17.4             & [11,13]     & 283      & [5.5,6]     & 286      \\
SiO               & 33401 & -1.68            & 84.4             & [2,4]       & 332      & [41,41.5]   & 332      \\
Andrews           & 60000 & -2.17e-14        & 36.5             & [1.5,2.5]   & 346      & [18.1,18.3] & 353      \\
\bottomrule
\end{tabular}
\end{table}

In what follows, since the CJ series requires the spectrum of $A$ to be in $[-1, 1]$,
the matrix $A$ and interval ends $a, b$ are replaced by $l(A)$, $l(a)$ and $l(b)$,
where $l(t) = (2t - \lambda_{\min} - \lambda_{\max})/(\lambda_{\max} - \lambda_{\min})$.
For all the test problems,
let the starting matrix $V^{(0)}$ be an orthonormalization of a Gaussian
random matrix generated from the standard normal distribution,
and the random matrix $C$ in the combination restart \cref{eq: def: combination restart} also be generated from the standard normal distribution.

\subsection{Numerical instability of the two natural restarts}

In this subsection, we numerically verify the sharpness of the lower bounds~\eqref{eq: ratio of singular values of S^(k) for SI-like} and \eqref{eq: ratio of singular values of S^(k) for SI-orth} for $\sigma_1(S^{(k)})/\sigma_{n_{ev}}(S^{(k)})$ in \cref{thm: ratio of singular values of S^(k) for SI-like,thm: ratio of singular values of S^(k) for SI-orth}.
We test the matrices SiH4, SiNa, delaunay\_n13, and benzene in \cref{tab: test matrices} with interval I's.
For each test problem,
we set the series degree $d = \lceil D \pi^2 (b-a)^{-4/3} \rceil - 2$ with $D = 1, 2$,
the moment numbers $M = 2, 4, 8$, and $M\ell = 8 \lceil 1.5 n_{ev} / 8 \rceil$.
For each of the problems, we run \cref{alg: CJSSRRR} with the natural restart schemes
\cref{eq: def: SI restart,eq: def: SI-orth restart} with the same $V^{(0)}$, respectively,
and perform $100$ restarts, i.e., all the $100$ restarts are displayed in the figures.
\cref{tab: ratio of singular values of S^(k) and real rates} reports the results after $30$ restarts.

\begin{figure}[ht]
  \centering
  \newcommand{\txtW}{0.47\textwidth}

  \begin{subfigure}[t]{\txtW}
    \centering
    \includegraphics[width=5.5cm, height=3.5cm]{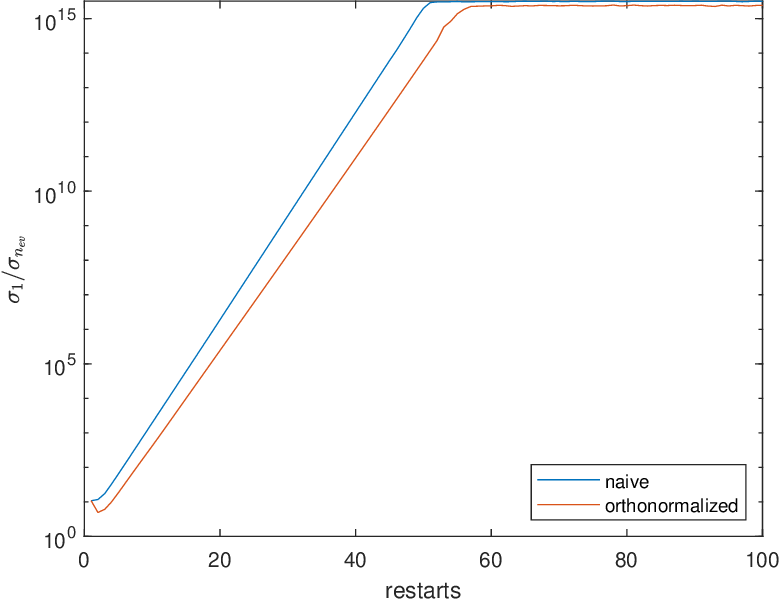}
    \subcaption{SiH4}
  \end{subfigure}
  \begin{subfigure}[t]{\txtW}
    \centering
    \includegraphics[width=5.5cm, height=3.5cm]{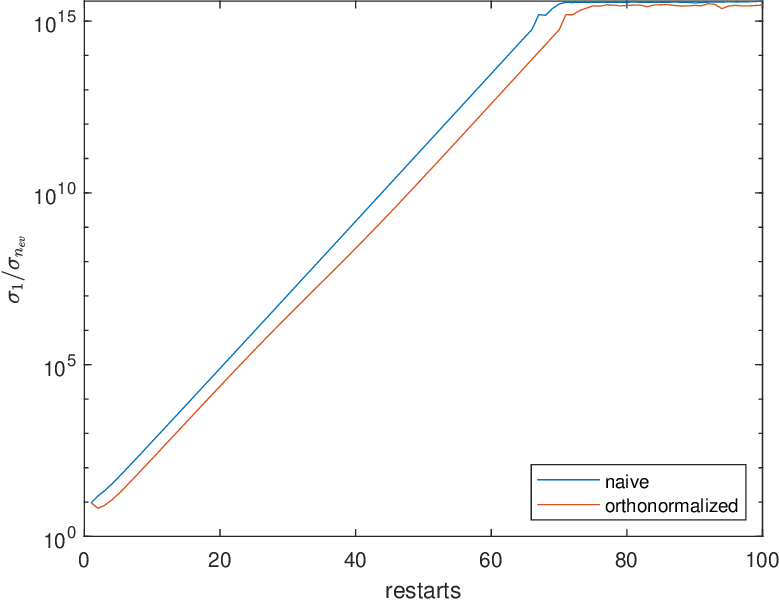}
    \subcaption{SiNa}
  \end{subfigure}

  \begin{subfigure}[t]{\txtW}
    \centering
    \includegraphics[width=5.5cm, height=3.5cm]{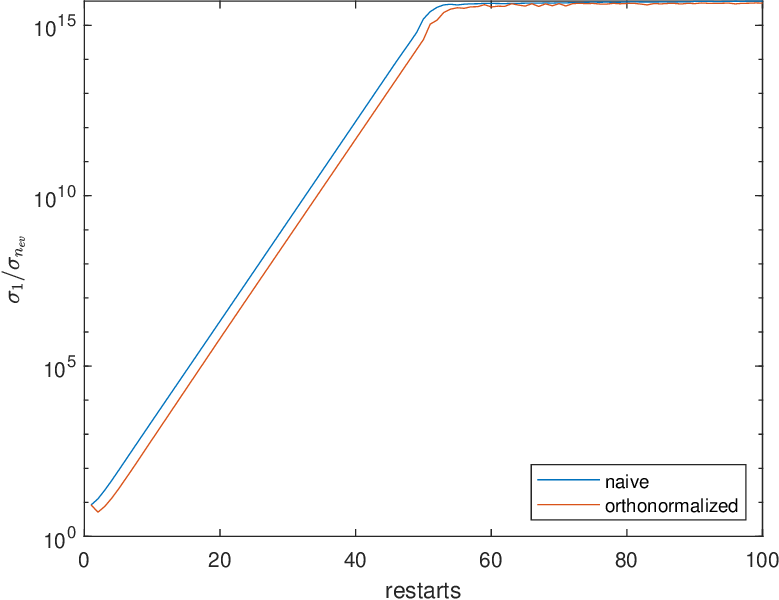}
    \subcaption{delaunay\_n13}
  \end{subfigure}
  \begin{subfigure}[t]{\txtW}
    \centering
    \includegraphics[width=5.5cm, height=3.5cm]{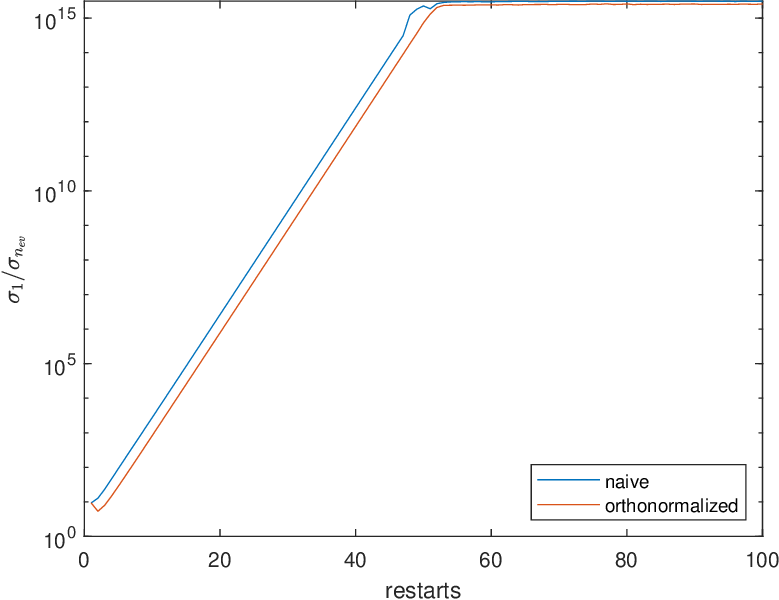}
    \subcaption{benzene}
  \end{subfigure}

  \caption{
    The ratios $\sigma_1(S^{(k)}) / \sigma_{n_{ev}}(S^{(k)})$ obtained by the naive and
    orthonormalized natural
    restarts with $D = 1$ and $M = 2$.
  }
  \label{fig: two SI restarts: ratio of singular values}
\end{figure}

\begin{figure}[ht]
  \centering
  \newcommand{\txtW}{0.47\textwidth}

  \begin{subfigure}[t]{\txtW}
    \centering
    \includegraphics[width=5.5cm, height=3.5cm]{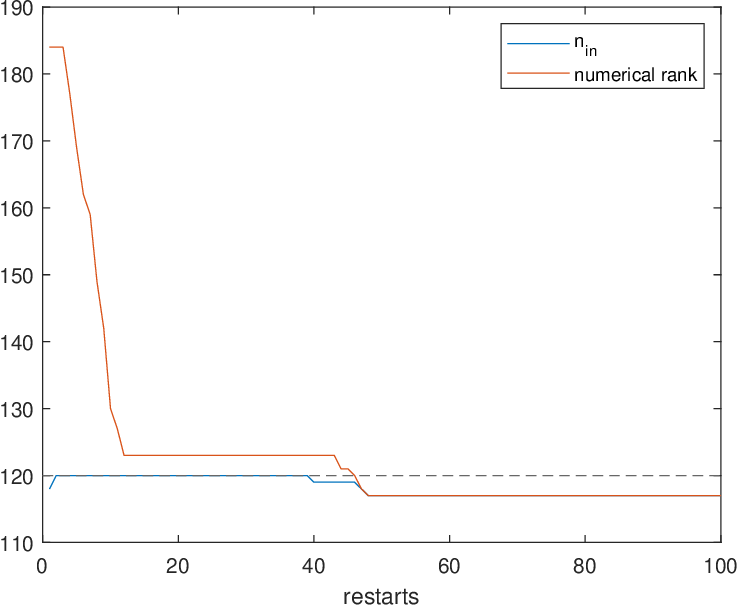}
    \subcaption{SiH4}
  \end{subfigure}
  \begin{subfigure}[t]{\txtW}
    \centering
    \includegraphics[width=5.5cm, height=3.5cm]{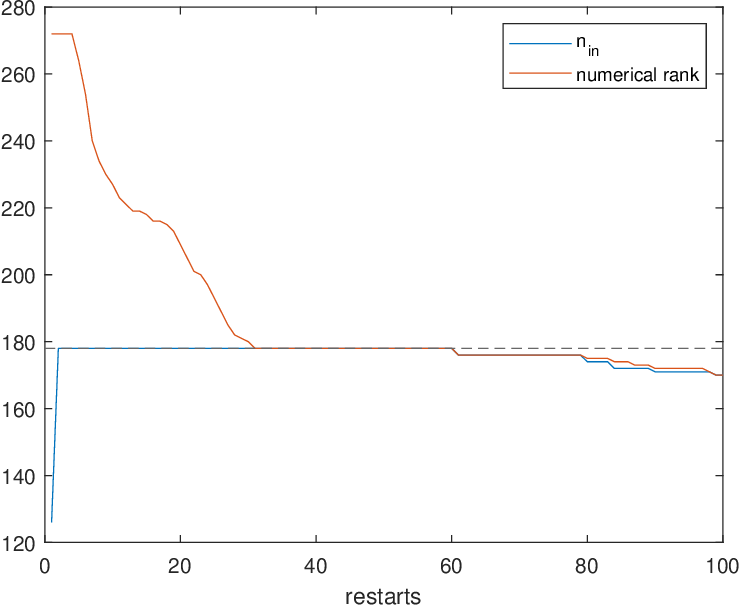}
    \subcaption{SiNa}
  \end{subfigure}

  \begin{subfigure}[t]{\txtW}
    \centering
    \includegraphics[width=5.5cm, height=3.5cm]{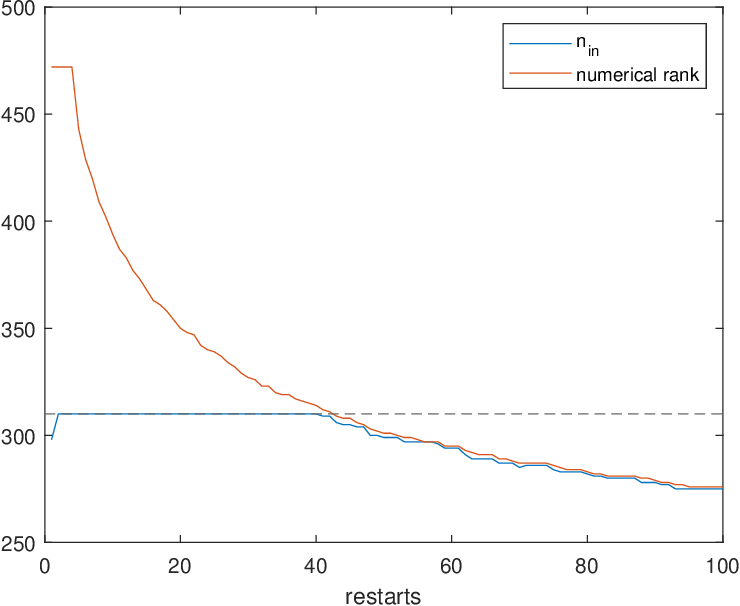}
    \subcaption{delaunay\_n13}
  \end{subfigure}
  \begin{subfigure}[t]{\txtW}
    \centering
    \includegraphics[width=5.5cm, height=3.5cm]{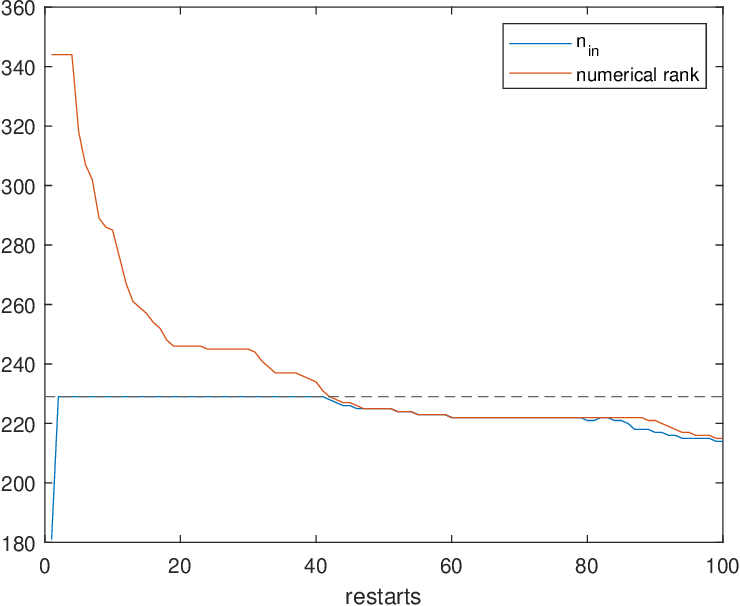}
    \subcaption{benzene}
  \end{subfigure}

  \caption{
    The number $n_{in}$ of refined Ritz values in $[a, b]$ and the numerical rank of $S^{(k)}$
    obtained by the orthonormalized natural restart with $D = 1$ and $M = 2$.
    The horizontal dashed lines indicate the number $n_{ev}$.
  }
  \label{fig: SI_orth restart: n_in and rank}
\end{figure}

\begin{table}[th]
\centering
\caption{
 The values of $\gamma_1 / \gamma_{n_{ev}}$ and $\gamma_\ell / \gamma_{n_{ev}}$,
$(\gamma_1 / \gamma_{n_{ev}})^{30}$ and $(\gamma_\ell / \gamma_{n_{ev}})^{30}$, and
$\sigma_1(S^{(30)}) / \sigma_{n_{ev}}(S^{(30)})$, where ``naive'' and ``orth''
denote the naive natural restart \cref{eq: def: SI restart} and the
orthonormalized natural restart \cref{eq: def: SI-orth restart}, respectively.
}
\label{tab: ratio of singular values of S^(k) and real rates}
\small
\begin{tabular}{*{9}{c}}
\toprule
$M$ & Matrix & $d$ & $\frac{\gamma_1}{\gamma_{n_{ev}}}$ & $\frac{\gamma_\ell}{\gamma_{n_{ev}}}$ &
 $(\frac{\gamma_1}{\gamma_{n_{ev}}})^{30}$ & $(\frac{\gamma_\ell}{\gamma_{n_{ev}}})^{30}$  & \multicolumn{2}{c}{$\frac{\sigma_1(S^{(30)})}{\sigma_{n_{ev}}(S^{(30)})}$} \\ \cline{8-9}
                   &                                &     &      &      & &   & naive & orth   \\
\midrule
\multirow{8}{*}{2} & \multirow{2}{*}{SiH4}          & 196 & 1.99 & 1.80 & 9.2e+8    &  4.6e+7    & 2.0e+9  & 1.5e+8 \\
                   &                                & 393 & 1.98 & 1.98 & 7.9e+8     & 7.9e+8    & 2.1e+9  & 5.7e+8 \\ \cline{2-9}
                   & \multirow{2}{*}{SiNa}          & 120 & 1.64 & 1.45 &  2.8e+6   & 6.9e+4    & 8.5e+6  & 2.4e+6 \\
                   &                                & 242 & 1.41 & 1.40 & 3.0e+4     &2.4e+4    & 8.4e+4  & 2.6e+4 \\ \cline{2-9}
                   & \multirow{2}{*}{delaunay\_n13} & 83  & 1.97 & 1.80 & 6.8e+8     & 4.6e+7   & 1.7e+9  & 4.6e+8 \\
                   &                                & 168 & 1.94 & 1.93 &4.3e+8     & 3.7e+8    & 1.0e+9  & 3.3e+8 \\ \cline{2-9}
                   & \multirow{2}{*}{benzene}       & 207 & 1.99 & 1.89 & 9.2e+8    & 2.0e+8    & 2.5e+9  & 6.4e+8 \\
                   &                                & 416 & 1.99 & 1.98 & 9.2e+8     & 7.9e+8   & 2.4e+9  & 7.2e+8 \\
\midrule
\multirow{8}{*}{4} & \multirow{2}{*}{SiH4}          & 196 & 1.99 & 1.99 & 9.2e+8     & 9.2e+8   & 2.2e+9  & 1.1e+9 \\
                   &                                & 393 & 1.98 & 1.98 & 7.9e+8    & 7.9e+8   & 1.8e+9  & 1.1e+9 \\ \cline{2-9}
                   & \multirow{2}{*}{SiNa}          & 120 & 1.64 & 1.64 & 2.8e+6     & 2.8e+6  & 1.2e+7  & 7.7e+6 \\
                   &                                & 242 & 1.41 & 1.41 & 3.0e+4     & 3.0e+4   & 9.8e+4  & 6.1e+4 \\ \cline{2-9}
                   & \multirow{2}{*}{delaunay\_n13} & 83  & 1.97 & 1.96 & 6.8e+8     & 5.9e+8   & 2.6e+9  & 1.3e+9 \\
                   &                                & 168 & 1.94 & 1.94 & 4.3e+8     & 4.3e+8   & 1.2e+9  & 6.7e+8 \\ \cline{2-9}
                   & \multirow{2}{*}{benzene}       & 207 & 1.99 & 1.99 & 9.2e+8     & 9.2e+8   & 3.0e+9  & 1.7e+9 \\
                   &                                & 416 & 1.99 & 1.99 & 9.2e+8     & 9.2e+8  & 2.4e+9  & 1.4e+9 \\
\midrule
\multirow{8}{*}{8} & \multirow{2}{*}{SiH4}          & 196 & 1.99 & 1.99 & 9.2e+8    & 9.2e+8  & 3.1e+9  & 2.3e+9 \\
                   &                                & 393 & 1.98 & 1.98 & 7.9e+8    & 7.9e+8  & 2.8e+9  & 2.3e+9 \\ \cline{2-9}
                   & \multirow{2}{*}{SiNa}          & 120 & 1.64 & 1.64 & 2.8e+6     &2.8e+6  & 5.5e+7  & 4.0e+7 \\
                   &                                & 242 & 1.41 & 1.41 & 3.0e+4     & 3.0e+4 & 1.9e+5  & 1.5e+5 \\ \cline{2-9}
                   & \multirow{2}{*}{delaunay\_n13} & 83  & 1.97 & 1.97 & 6.8e+8     &6.8e+8  & 5.3e+9  & 3.8e+9 \\
                   &                                & 168 & 1.94 & 1.94 & 4.3e+8    & 4.3e+8  & 2.1e+9  & 1.6e+9 \\ \cline{2-9}
                   & \multirow{2}{*}{benzene}       & 207 & 1.99 & 1.99 & 9.2e+8     & 9.2e+8  & 8.0e+9  & 6.2e+9 \\
                   &                                & 416 & 1.99 & 1.99 & 9.2e+8     & 9.2e+8  & 3.3e+9  & 2.7e+9 \\
\bottomrule
\end{tabular}
\end{table}

\Cref{fig: two SI restarts: ratio of singular values} shows the ratios $\sigma_1(S^{(k)}) / \sigma_{n_{ev}}(S^{(k)})$ obtained by them with $D = 1$ and $M = 2$.
It is seen that for all the test problems,
the ratio $\sigma_1(S^{(k)}) / \sigma_{n_{ev}}(S^{(k)})$ obtained by \cref{eq: def: SI-orth restart} was always smaller than that obtained by \cref{eq: def: SI restart} at each restart.
This indicates that the orthonormalized natural restart was numerically more stable
than the naive natural restart though the two natural restarts were both numerically unstable.

\Cref{fig: SI_orth restart: n_in and rank} shows the number $n_{in}$ of refined Ritz
values in $[a,b]$ and the numerical rank of $S^{(k)}$ obtained by
the orthonormalized natural restart with $D = 1$ and $M = 2$,
where the numerical rank is computed by the TSVD with the threshold $100\epsilon_{\rm mach}\|S^{(k)}\|_F$.
For each problem, the numerical rank was initially equal to  $M\ell$, the column number of $S^{(k)}$,
then decreased with $k$, and finally dropped below $n_{ev}$. The number $n_{in}$ may be smaller than
$n_{ev}$ before the first restart due to the inaccuracy of the initial subspace $\Range(S^{(0)})$,
then was equal to $n_{ev}$ as the restarts proceeded; afterwards it decreased,
together with the numerical rank of $S^{(k)}$, and ultimately dropped below $n_{ev}$.
In other words, more and more desired eigenpairs were lost as
the divergence proceeds because of numerical instability of the natural restart.
These confirm Remark \ref{rem3.1} and the first paragraph in \cref{section:3.2}.

\cref{tab: ratio of singular values of S^(k) and real rates} displays
the results after $30$ restarts by the two natural restarts with different $M$ and $d$, where
we list the values of $\gamma_1 / \gamma_{n_{ev}}$ and $\gamma_\ell / \gamma_{n_{ev}}$,
$(\gamma_1 / \gamma_{n_{ev}})^{30}$ and $(\gamma_\ell / \gamma_{n_{ev}})^{30}$, and
$\sigma_1(S^{(30)}) / \sigma_{n_{ev}}(S^{(30)})$.
We observe from \cref{tab: ratio of singular values of S^(k) and real rates} that the values of
$(\gamma_1 / \gamma_{n_{ev}})^{30}$ and $(\gamma_\ell / \gamma_{n_{ev}})^{30}$ were very
good estimates for $\sigma_1(S^{(30)}) / \sigma_{n_{ev}}(S^{(30)})$, confirming
the sharpness of lower bounds in \cref{thm: ratio of singular values of S^(k) for SI-like,thm: ratio of singular values of S^(k) for SI-orth}.
We also see from \cref{tab: ratio of singular values of S^(k) and real rates} that
the ratio $\sigma_1(S^{(30)}) / \sigma_{n_{ev}}(S^{(30)})$ obtained by the orthonormalized natural
restart was always smaller than that obtained by the naive natural restart for each of the test problems
with different $M$ and $d$, justifying
that the orthonormalized natural restart had better numerical stability
than the naive natural restart. We should point out that,
when the QR decomposition is employed to orthonormalize $S^{(k)}$,
the orthonormalization of $PV^{(k)}$ is also available simultaneously without additional
computational cost.
Therefore, of the two natural restarts,
the orthonormalized natural restart is preferable,
and in the sequel, the natural restart will mean \cref{eq: def: SI-orth restart}.

\subsection{The comparison of the natural restart and combination restart}
\label{subsec: comparison of SI and thick-like}

In this subsection, we compare the performances of the CJ--SS--RRR algorithms
with the natural restart \cref{eq: def: SI-orth restart} and combination restart \cref{eq: def: combination restart}.
The test matrices and intervals are the same as those in the last subsection.
We take the moment number $M=8$,
and choose the series degree $d$ as in \cref{eq: degree estimate for practice} with $D = 1$, and
$M\ell\approx 1.5 n_{ev} $. For each test problem,
we still perform $30$ restarts for the two restart schemes using the same starting matrix $V^{(0)}$.

\begin{figure}[t]
  \centering
  \newcommand{\txtW}{0.47\textwidth}

  \begin{subfigure}[t]{\txtW}
    \centering
    \includegraphics[width=5.5cm, height=3.5cm]{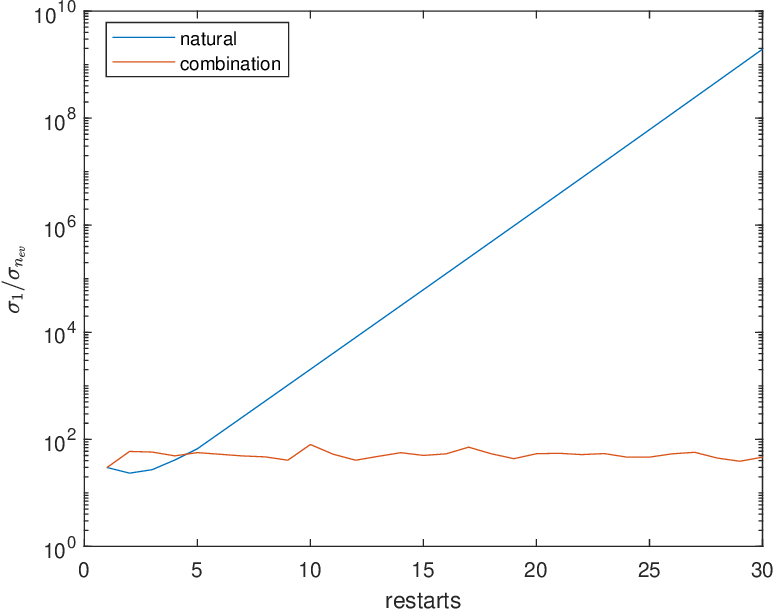}
    \subcaption{SiH4}
  \end{subfigure}
  \begin{subfigure}[t]{\txtW}
    \centering
    \includegraphics[width=5.5cm, height=3.5cm]{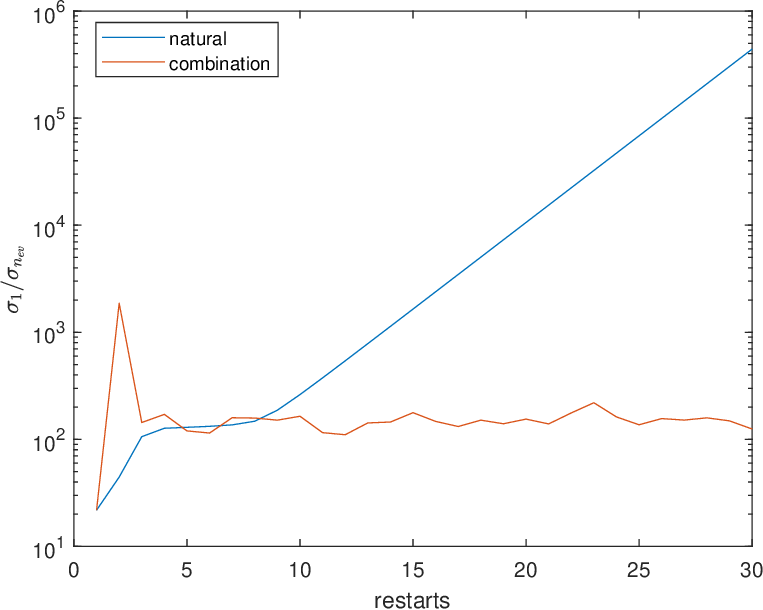}
    \subcaption{SiNa}
  \end{subfigure}

  \begin{subfigure}[t]{\txtW}
    \centering
    \includegraphics[width=5.5cm, height=3.5cm]{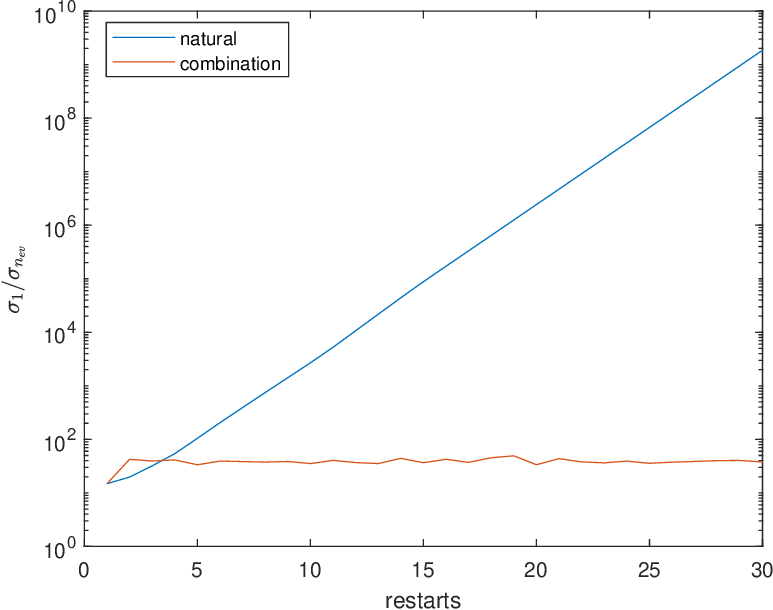}
    \subcaption{delaunay\_n13.}
  \end{subfigure}
  \begin{subfigure}[t]{\txtW}
    \centering
    \includegraphics[width=5.5cm, height=3.5cm]{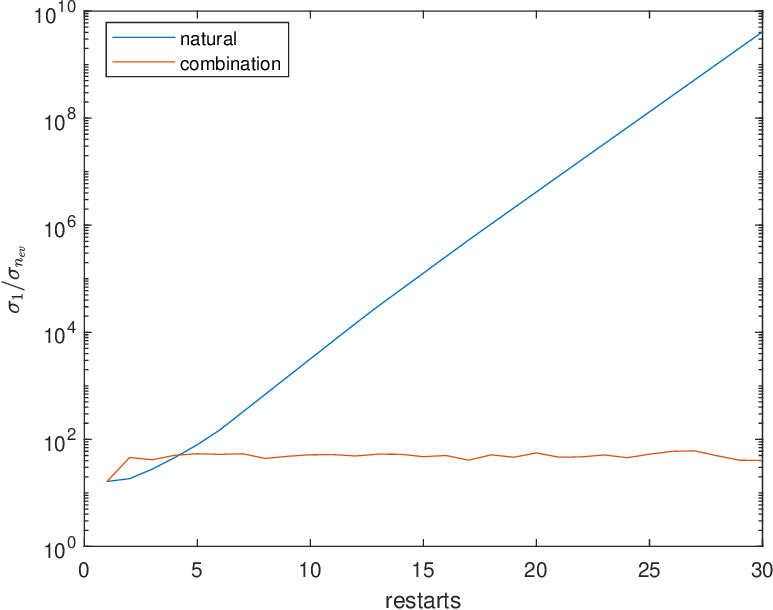}
    \subcaption{benzene}
  \end{subfigure}

  \caption{
    The ratios $\sigma_1(S^{(k)}) / \sigma_{n_{ev}}(S^{(k)})$ obtained by the natural and combination restarts.
  }
  \label{fig: SI vs thick-like: ratio of singular values}
\end{figure}

We plot the ratios $\sigma_1(S^{(k)}) / \sigma_{n_{ev}}(S^{(k)})$ obtained by the natural and combination restarts in \cref{fig: SI vs thick-like: ratio of singular values}.
As we can see, for the four test problems,
the ratios obtained by the combination restart remained almost unchanged and
are around $100$ after a few restarts at most,
whereas the ones obtained by the natural restart increased geometrically with $k$.
These numerical results confirm \cref{thm: ratio of singular values of S^(k) for SI-like,thm: ratio of singular values of S^(k) for SI-orth,idealkappa}.

\begin{figure}[t]
  \centering
  \newcommand{\txtW}{0.47\textwidth}

  \begin{subfigure}[t]{\txtW}
    \centering
    \includegraphics[width=5.5cm, height=3.5cm]{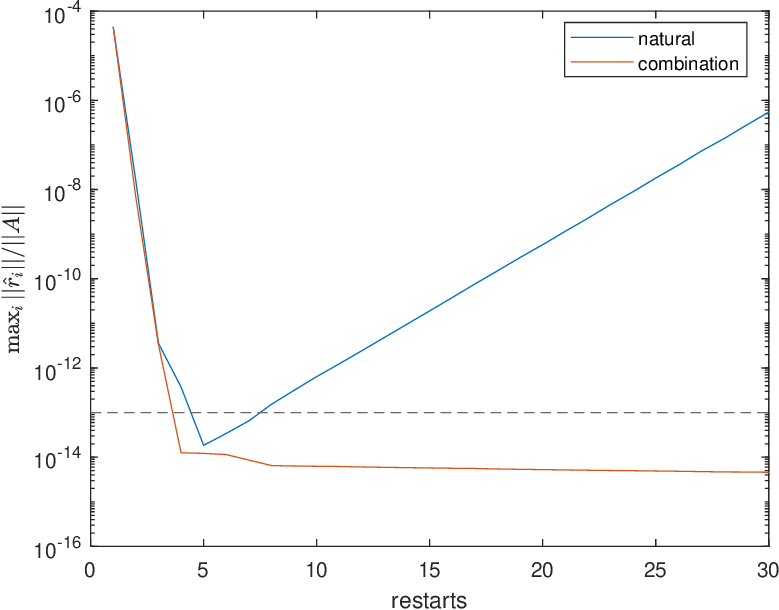}
    \subcaption{SiH4}
  \end{subfigure}
  \begin{subfigure}[t]{\txtW}
    \centering
    \includegraphics[width=5.5cm, height=3.5cm]{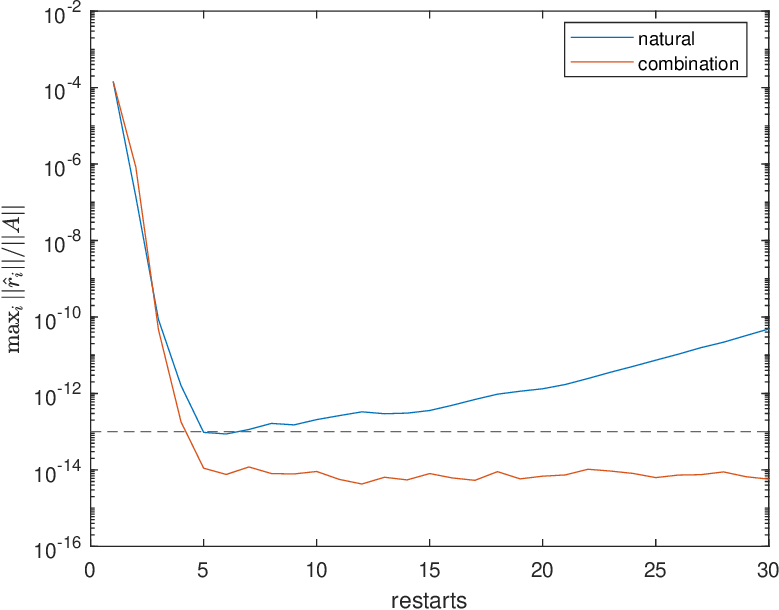}
    \subcaption{SiNa}
  \end{subfigure}

  \begin{subfigure}[t]{\txtW}
    \centering
    \includegraphics[width=5.5cm, height=3.5cm]{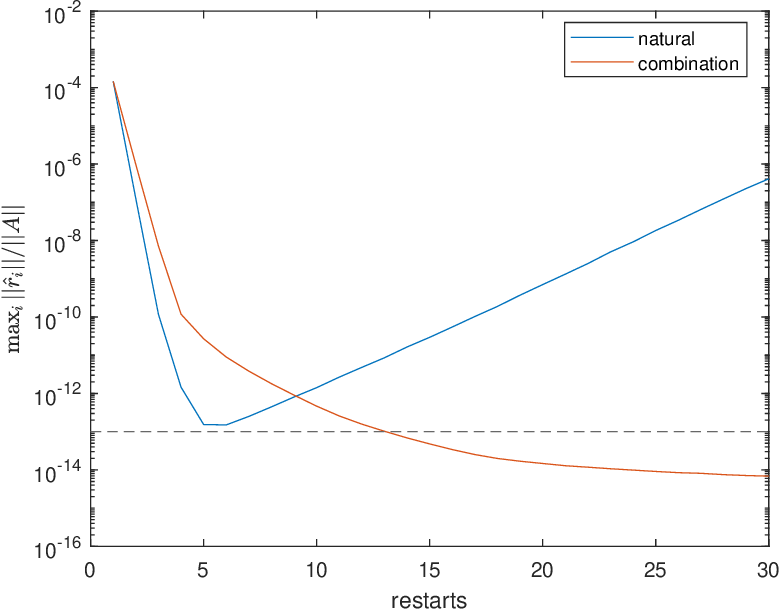}
    \subcaption{delaunay\_n13}
  \end{subfigure}
  \begin{subfigure}[t]{\txtW}
    \centering
    \includegraphics[width=5.5cm, height=3.5cm]{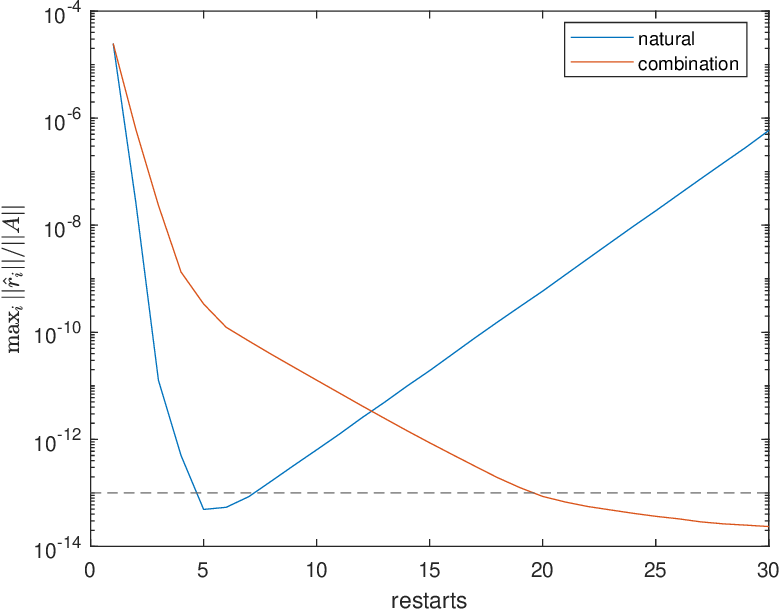}
    \subcaption{benzene}
  \end{subfigure}

  \caption{
    The maximum relative residual norms of the refined Ritz pairs with the refined Ritz values
    in given intervals obtained by the natural
    and combination restarts,
    where the dashed lines indicate the convergence tolerance $tol=10^{-13}$.
  }
  \label{fig: SI vs thick-like: residual norms}
\end{figure}

\cref{fig: SI vs thick-like: residual norms} draws
the maximum relative residual norms of the refined Ritz pairs with
the refined Ritz values in $[a,b]$. As is seen,
for the convergence tolerance $tol = 10^{-13}$,
the algorithm with the natural restart failed
to converge for the test problem delaunay\_n13. We find that it first converged regularly
until the residual norms were at the level of $10^{-13}$ or slightly larger, and
then it started to diverge. This is due to the fact
that $\sigma_1(S^{(k)}) / \sigma_{n_{ev}}(S^{(k)})$ increased quickly and the search subspace obtained
by its QR decomposition severely deviated from the true $\Range(S^{(k)})$ in later restarts.
By contrast, for the four test problems,
the algorithm with the combination restart converged to the working precision
since $\sigma_1(S^{(k)}) / \sigma_{n_{ev}}(S^{(k)})$
was constantly modest.

These numerical results reveal the intrinsic effects of the size of $\sigma_1(S^{(k)}) / \sigma_{n_{ev}}(S^{(k)})$ on the attainable relative
residual norms of approximate eigenpairs,
and confirm the superiority of the combination restart in terms of numerical stability.
They demonstrate the possible failure of the algorithm with the natural restart when greater
accuracy is desired, and justify that
the combination restart is a safer and more effective choice than the natural restart.

\subsection{The performance of the hybrid restart}

As \cref{fig: SI vs thick-like: residual norms} shows, the CJ--SS--RRR
algorithms with the natural and combination
restarts converged equally fast for the problems SiH4 and SiNa in the first stage, but the CJ--SS--RRR algorithm
with the combination restart was slower than the counterpart with
the natural restart for delaunay\_n13 and benzene.
In this subsection, we test the overall efficiency of \cref{alg: CJSSRRR_hybrid},
i.e., the CJ--SS--RRR algorithm with the hybrid restart,
where we take several $\epsilon_{\rm switch}>tol$ in \cref{eq: switch criterion}.
The other parameters are the same as in the last subsection.

As a comparison, we also report the results using the CJ--SS--RRR
algorithms with the natural restart and the combination restart,
respectively.
We test the first six matrices SiH4, SiNa, delaunay\_n13, benzene, stokes64, and Pres\_Poisson in
\cref{tab: test matrices} with interval I's. For each test problem,
we run each of the algorithms five times with five different starting matrices $V^{(0)}$.
In \cref{fig: hybrid: residual norms}, we display the average convergence histories over the first $30$ restarts.

\begin{figure}[t]
  \centering
  \newcommand{\txtW}{0.32\textwidth}

  \begin{subfigure}[t]{\txtW}
    \centering
    \includegraphics[width=\textwidth]{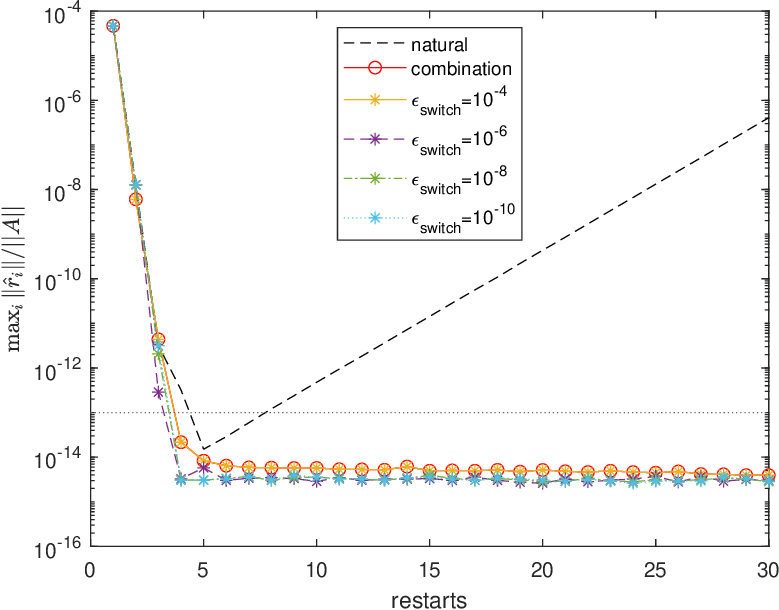}
    \subcaption{SiH4.}
  \end{subfigure}
  \begin{subfigure}[t]{\txtW}
    \centering
    \includegraphics[width=\textwidth]{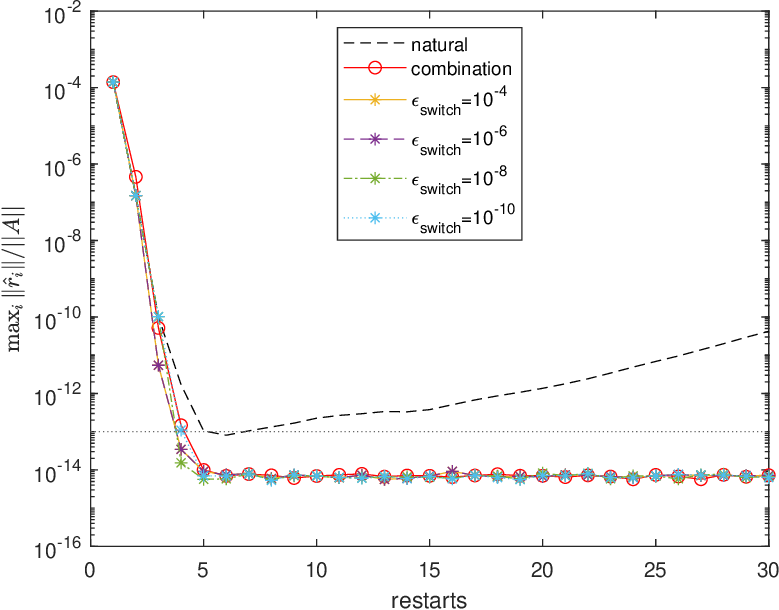}
    \subcaption{SiNa.}
  \end{subfigure}
  \begin{subfigure}[t]{\txtW}
    \centering
    \includegraphics[width=\textwidth]{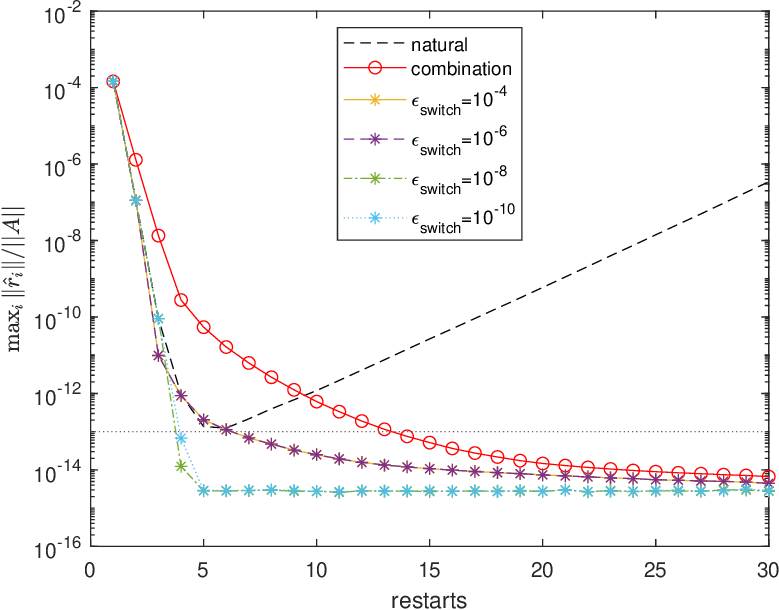}
    \subcaption{delaunay\_n13}
  \end{subfigure}

  \begin{subfigure}[t]{\txtW}
    \centering
    \includegraphics[width=\textwidth]{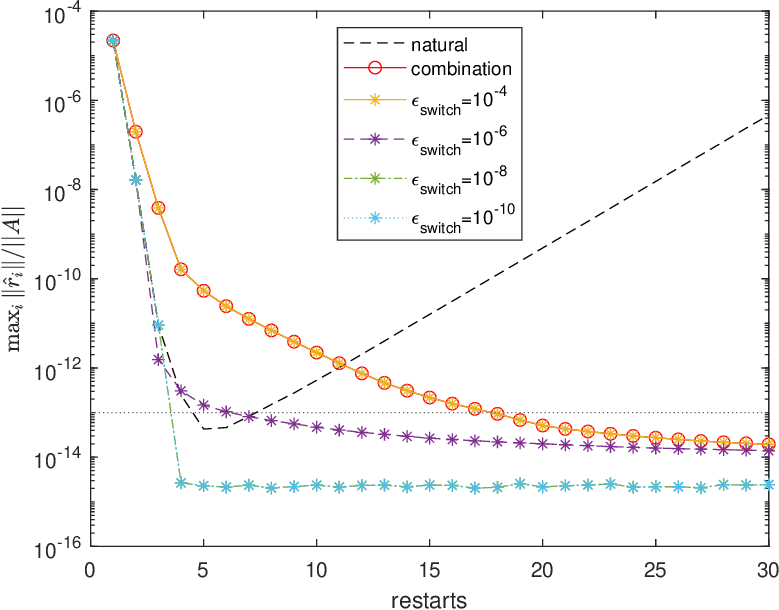}
    \subcaption{benzene}
  \end{subfigure}
  \begin{subfigure}[t]{\txtW}
    \centering
    \includegraphics[width=\textwidth]{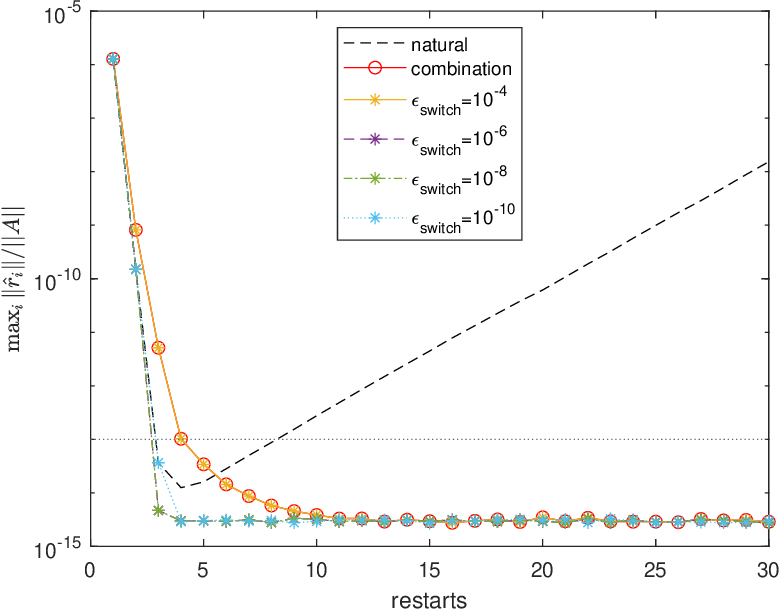}
    \subcaption{stokes64}
  \end{subfigure}
  \begin{subfigure}[t]{\txtW}
    \centering
    \includegraphics[width=\textwidth]{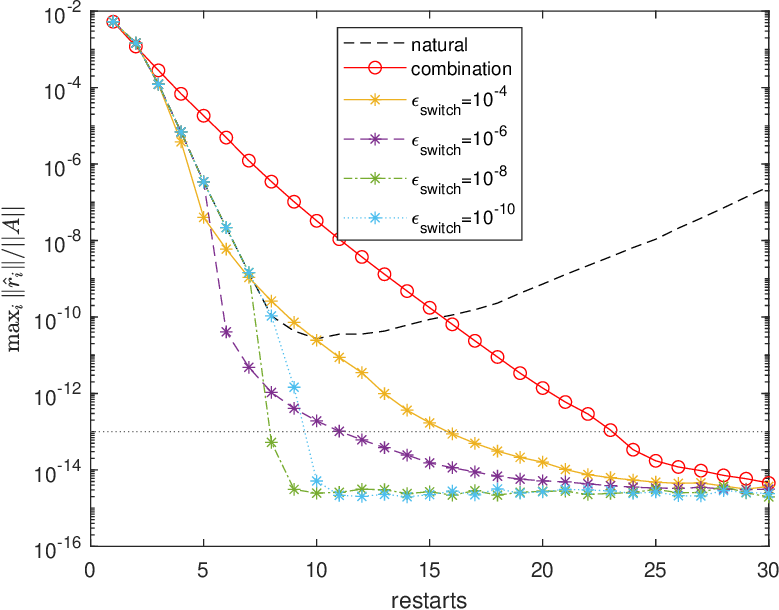}
    \subcaption{Pres\_Poisson}
    \label{subfig: hybrid: residual norms: Pres_Poisson}
  \end{subfigure}

  \caption{
    The average results over $5$ runs,
    where the $y$-axis represents the maximum relative residual norms of the refined
    Ritz pairs with the refined Ritz values in the given intervals. 
    The curves labeled ``natural'' and ``combination'' represent the
    algorithms with the natural and combination restarts, respectively.
    The other curves represent the hybrid restart with different values of $\epsilon_{\rm switch}$.
    The dashed lines parallel to the $x$-axis indicate the convergence tolerance $tol=10^{-13}$.
  }
  \label{fig: hybrid: residual norms}
\end{figure}

\cref{fig: hybrid: residual norms} illustrates the average results over five runs,
where the $y$-axis represents the maximum relative residual norms of the refined
Ritz pairs with the refined Ritz values in $[a,b]$. As the figure indicates,
the algorithm with the combination restart converged for all six test problems for $tol=10^{-13}$,
though it required more restarts for delaunay\_n13, benzene, and Pres\_Poisson.
For these three problems,
the CJ--SS--RRR algorithm with the hybrid restart reduced the number of restarts required for convergence,
indicating that the hybrid restart was more efficient
when the combination restart converged slowly. Importantly, we observe that
the natural restart was not numerically backward stable to
the tolerance $10^{-13}$ for delaunay\_n13 and Pres\_Poisson.

\begin{table}[htbp]
\caption{
  The average restarts for convergence over $5$ runs, where
  the columns ``Natural'' and ``Combination'' represent the algorithms with the natural restart and
  the combination restart, respectively,
  the columns ``$\epsilon_{\rm switch}$'' represent \cref{alg: CJSSRRR_hybrid} with
  $\epsilon_{\rm switch}$ in \cref{eq: switch criterion},
  and ``-'' indicates that none of the $5$ runs converged within $100$ restarts.
  The results in bold indicate the best performance among the restart schemes, and ``*'' denotes no implementation because of $\epsilon_{\rm switch}\leq tol$.}
\label{tab: hybrid: number of restarts}
\centering

\setlength{\tabcolsep}{5pt}
\begin{tabular}{*{8}{c}}
\toprule
$tol$                       & Matrix        & Natural    & Combination & \multicolumn{4}{c}{$\epsilon_{\rm switch}$}             \\ \cline{5-8}
                            &               &            &             & $10^{-4}$  & $10^{-6}$    & $10^{-8}$  & $10^{-10}$ \\
\midrule
\multirow{6}{*}{$10^{-8}$}  & SiH4          & 2.8        & \textbf{2}  & \textbf{2} & 2.8          & *          & *          \\
                            & SiNa          & \textbf{3} & \textbf{3}  & \textbf{3} & \textbf{3}   & *          & *          \\
                            & delaunay\_n13 & \textbf{3} & 3.8         & \textbf{3} & \textbf{3}   & *          & *          \\
                            & benzene       & \textbf{3} & 3.2         & 3.2        & \textbf{3}   & *          & *          \\
                            & stokes64      & \textbf{2} & \textbf{2}  & \textbf{2} & \textbf{2}   & *          & *          \\
                            & Pres\_Poisson & 7          & 11.4        & 6.2        & \textbf{6}   & *          & *          \\
\midrule
\multirow{6}{*}{$10^{-10}$} & SiH4          & \textbf{3} & \textbf{3}  & \textbf{3} & \textbf{3}   & \textbf{3} & *          \\
                            & SiNa          & 3.4        & \textbf{3}  & \textbf{3} & \textbf{3}   & 3.4        & *          \\
                            & delaunay\_n13 & 3.4        & 5           & \textbf{3} & \textbf{3}   & 3.4        & *          \\
                            & benzene       & \textbf{3} & 5           & 5          & \textbf{3}   & \textbf{3} & *          \\
                            & stokes64      & \textbf{3} & \textbf{3}  & \textbf{3} & \textbf{3}   & \textbf{3} & *          \\
                            & Pres\_Poisson & 8.6        & 16.2        & 9.6        & \textbf{6}   & 8          & *          \\
\midrule
\multirow{6}{*}{$10^{-12}$} & SiH4          & 4          & 4           & 4          & \textbf{3}   & 3.8        & 4          \\
                            & SiNa          & 5          & \textbf{4}  & \textbf{4} & \textbf{4}   & \textbf{4} & 4.4        \\
                            & delaunay\_n13 & 4.4        & 9.8         & 4.4        & 4.4          & \textbf{4} & \textbf{4} \\
                            & benzene       & 4          & 11.8        & 11.8       & \textbf{3.8} & 4          & 4          \\
                            & stokes64      & \textbf{3} & 4           & 4          & \textbf{3}   & \textbf{3} & \textbf{3} \\
                            & Pres\_Poisson & -          & 21          & 13.4       & 8.6          & \textbf{8} & 9.6        \\
\midrule
\multirow{6}{*}{$10^{-14}$} & SiH4          & -          & 6           & 6          & \textbf{4}   & \textbf{4} & \textbf{4} \\
                            & SiNa          & -          & 5.6         & 5.4        & 5.4          & \textbf{5} & 5.2        \\
                            & delaunay\_n13 & -          & 24          & 16.4       & 16.4         & \textbf{5} & \textbf{5} \\
                            & benzene       & -          & 55.4        & 55.4       & 45.4         & \textbf{4} & \textbf{4} \\
                            & stokes64      & -          & 7.4         & 7.4        & \textbf{3}   & \textbf{3} & 4          \\
                            & Pres\_Poisson & -          & 27.4        & 20.8       & 17           & \textbf{9} & 10         \\
\bottomrule
\end{tabular}

\end{table}

\cref{tab: hybrid: number of restarts} lists the average
restarts needed to reach different $tol$'s, 
where ``-'' denotes failure to converge in all five runs.
From the table, we find that,
compared to the combination restart,
using the hybrid restart with $\epsilon_{\rm switch} \approx \sqrt{tol}$
reduced the number of restarts substantially,
especially when $tol$ is small.
In particular, for $tol=10^{-14}$,
the algorithm with the natural restart failed to converge when $100$ restarts were consumed;
for Pres\_Poisson, the algorithm with the natural restart even failed for $tol=10^{-12}$.

Finally, regarding the overall efficiency, we see from the numerical results that
$\epsilon_{\rm switch} \approx \sqrt{tol}$ was a good choice.

\subsection{A comparison of CJ--SS--RRR with the natural restart and CJ--SS--RRR with the hybrid restart}

In this subsection, we compare the performance of \cref{alg: CJSSRRR}
and \cref{alg: CJSSRRR_hybrid} when greater accuracy is desired: $tol=10^{-13}$ and $10^{-14}$,
which are (almost) at the level of $\epsilon_{\rm mach}$ and require an algorithm to be
(almost) numerically backward stable.
We test the problems with intervals of interest in \cref{tab: test matrices},
and take the same parameters as in the last two subsections.
For each test problem and each $tol$,
we run each of the two algorithms five times with five different starting matrices $V^{(0)}$, allowing a maximum of $100$ restarts.

\begin{table}[!htbp]
\caption{
  The average restarts for convergence over $5$ runs and $tol=10^{-13}$ and $10^{-14}$, respectively.
  The mark ``-'' indicates failure when 100 restarts were used.
  The results with parentheses represent the percentage of successful convergence among 5 runs.
  The results in bold indicate the best performance.
}
\label{tab: comparison of CJSSRRR and CJSSRRR-thick}
\centering
\setlength{\tabcolsep}{5pt}

\begin{tabular}{*{10}{c}}
\toprule
\multirow{2}{*}{Interval} & \multirow{2}{*}{Matrix}
& \multicolumn{2}{c}{\multirow{2}{*}{\cref{alg: CJSSRRR}}} & & \multicolumn{5}{c}{\cref{alg: CJSSRRR_hybrid}} \\
\cline{6-10}
& & \multicolumn{2}{c}{} & & \multicolumn{2}{c}{Combination} & & \multicolumn{2}{c}{$\epsilon_{\rm switch} = 10^{-7}$} \\
\cline{3-4} \cline{6-7} \cline{9-10}
                     &                   & $10^{-13}$ & $10^{-14}$ & & $10^{-13}$ & $10^{-14}$ & & $10^{-13}$   & $10^{-14}$ \\
\midrule
\multirow{10}{*}{I}  & SiH4              & 5            & -        & & 4.2        & 6          & & \textbf{4}   & \textbf{4}    \\
                     & SiNa              & 5.75 (80\%)  & -        & & 5          & 5.6        & & \textbf{4}   & \textbf{5}    \\
                     & delaunay\_n13     & -            & -        & & 13.8       & 24         & & \textbf{4.6} & \textbf{7.6}  \\
                     & benzene           & \textbf{5}   & -        & & 18         & 55.4       & & 6.8          & \textbf{45.4} \\
                     & stokes64          & \textbf{3}   & -        & & 4.6        & 7.4        & & \textbf{3}   & \textbf{3}    \\
                     & Pres\_Poisson     & -            & -        & & 23.4       & 27.4       & & \textbf{8.2} & \textbf{13}   \\
                     & Si10H16           & 4            & -        & & 5          & 6.6        & & \textbf{3.2} & \textbf{4}    \\
                     & rgg\_n\_2\_15\_s0 & -            & -        & & 12.6       & 20.8       & & \textbf{5.6} & \textbf{13.2} \\
                     & SiO               & 3.2          & -        & & 3.4        & 4          & & \textbf{3}   & \textbf{3}    \\
                     & Andrews           & 5            & -        & & 9.2        & 16.8       & & \textbf{4}   & \textbf{9}    \\
\midrule
\multirow{10}{*}{II} & SiH4              & 4            & -        & & 4.2        & 5.4        & & \textbf{3.2} & \textbf{4.2}  \\
                     & SiNa              & 4.2          & -        & & 6.4        & 8.4        & & \textbf{4}   & \textbf{5}    \\
                     & delaunay\_n13     & -            & -        & & 33         & 61.8       & & \textbf{5.4} & \textbf{33.4} \\
                     & benzene           & \textbf{4.2} & -        & & 24.4       & 45.4       & & 10.4         & \textbf{33}   \\
                     & stokes64          & \textbf{5}   & -        & & 8          & 9.8        & & \textbf{5}   & \textbf{5}    \\
                     & Pres\_Poisson     & \textbf{4.2} & -        & & 9.6        & 12.4       & & 4.8          & \textbf{8.2}  \\
                     & Si10H16           & \textbf{4}   & -        & & 5          & 7          & & \textbf{4}   & \textbf{5}    \\
                     & rgg\_n\_2\_15\_s0 & \textbf{5}   & -        & & 13.6       & 18         & & 5.6          & \textbf{12}   \\
                     & SiO               & 4            & -        & & 4.4        & 7.6        & & \textbf{3}   & \textbf{4}    \\
                     & Andrews           & 4            & -        & & 6.4        & 12         & & \textbf{3}   & \textbf{4.2}  \\
\bottomrule
\end{tabular}
\end{table}

In \cref{tab: comparison of CJSSRRR and CJSSRRR-thick}, we report the average restarts
with $\epsilon_{\rm switch} = 10^{-7}\approx \sqrt{tol}$.
As the results indicate, for $tol=10^{-13}$,
\cref{alg: CJSSRRR} failed for four test cases in all five runs and
failed for one additional test case in one of the five runs.
For $tol=10^{-14}$, it failed to converge for all the test cases.

In contrast, the algorithms with the combination and hybrid restarts were successful
for all the test cases and the two $tol$'s.
Both algorithms were numerically backward stable to the working precision.
Moreover, for $tol=10^{-13}$,
the hybrid restart used fewer restarts than the combination restart in all the $20$ test cases.

\section{Conclusion}
\label{sec: conclusion}

We have investigated the naive natural restart used in the block SS--RR algorithms
including the CJ--SS--RR and CJ--SS--RRR algorithms, and proposed
an orthonormalized natural restart. We have proved that the resulting
two basis matrices of
the search subspace become increasingly ill-conditioned quickly as restarts proceed.
As a consequence, the corresponding restarted algorithms are numerically unstable, and cannot
compute the desired eigenpairs accurately; if computing the TSVD of the basis matrices to
obtain effective subspaces, then their dimensions ultimately are smaller than $n_{ev}$,
so that the algorithms cannot compute all the $n_{ev}$ desired eigenpairs. We have confirmed these
results numerically.

To fix the deficiency of natural restarts, we have proposed a combination restart scheme that fully
exploits the current approximate eigenvectors and their random linear
combinations to construct a better restarting matrix $V^{(k)}$.
We have shown that the numerical rank of the basis matrices of the new updated
subspaces is no smaller than $n_{ev}$, so that the subspaces retain the information
on all the desired eigenvectors and the algorithm is numerically stable.
To enhance the overall efficiency, we have further proposed a hybrid restart scheme
that starts with the natural restart and then
switches to the combination restart once the approximate eigenvectors start to have some accuracy.

We have conducted numerical experiments to confirm the results and analysis,
illustrating that the combination restart is numerically much more stable than the two natural restarts
and the algorithm with the hybrid restart is not only more efficient than the algorithm
with the natural restart but also
numerically backward stable to the working precision.

The combination and hybrid restarts, results and analysis can be extended to the block
SS--RR and SS--RRR for general standard and generalized eigenvalue problems, with some necessary
modifications.

\section*{Declarations}
The two authors declare that they have no
financial interests, and they read and approved the final manuscript.  

\smallskip

{\bf Data Availability} \ Not applicable.


\begin{thebibliography}{10}

\bibitem{bai2000}
{\sc Z.~Bai, J.~Demmel, J.~Dongarra, A.~Ruhe, and H.~A. Van~der Vorst}, {\em
  Templates for the Solution of Algebraic Eigenvalue Problems: A Practical
  Guide}, SIAM, Philadelphia, PA, 2000,
  \url{https://doi.org/10.1137/1.9780898719581}.

\bibitem{LSEP}
{\sc Z.~Bai and R.-C. Li}, {\em Minimization principles for the linear response
  eigenvalue problem {I}: Theory}, SIAM J. Matrix Anal. Appl., 33 (2012),
  pp.~1075--1100, \url{https://doi.org/10.1137/110838960}.

\bibitem{beckermann2000}
{\sc B.~Beckermann}, {\em The condition number of real {Vandermonde}, {Krylov}
  and positive definite {Hankel} matrices}, Numer. Math., 85 (2000),
  pp.~553--577, \url{https://doi.org/10.1007/PL00005392}.

\bibitem{matrix_collection}
{\sc T.~A. Davis and Y.~Hu}, {\em The {U}niversity of {Florida} sparse matrix
  collection}, ACM Trans. Math. Softw., 38 (2011),
  \url{https://doi.org/10.1145/2049662.2049663}.

\bibitem{eig_count_Saad}
{\sc E.~Di~Napoli, E.~Polizzi, and Y.~Saad}, {\em Efficient estimation of
  eigenvalue counts in an interval}, Numer. Linear Algebra Appl., 23 (2016),
  pp.~674--692, \url{https://doi.org/10.1002/nla.2048}.

\bibitem{IFEAST}
{\sc B.~Gavin and E.~Polizzi}, {\em {Krylov} eigenvalue strategy using the
  {FEAST} algorithm with inexact system solves}, Numer. Linear Algebra Appl.,
  25 (2018), e2188, \url{https://doi.org/10.1002/nla.2188}.

\bibitem{Golub_Matrix}
{\sc G.~H. Golub and C.~F. Van~Loan}, {\em Matrix Computations}, The Johns
  Hopkins University Press, Baltimore, 4th~ed., 2013,
  \url{https://doi.org/10.1137/1.9781421407944}.

\bibitem{FEAST_Zolotarev}
{\sc S.~G\"{u}ttel, E.~Polizzi, P.~T.~P. Tang, and G.~Viaud}, {\em {Zolotarev}
  quadrature rules and load balancing for the {FEAST} eigensolver}, SIAM J.
  Sci. Comput., 37 (2015), pp.~A2100--A2122,
  \url{https://doi.org/10.1137/140980090}.

\bibitem{higham}
{\sc N.~J. Higham}, {\em Accuracy and Stability of Numerical Algorithms, Second
  Edition}, SIAM, Philadelphia, PA, 2002,
  \url{https://doi.org/10.1137/1.9780898718027}.

\bibitem{block_SS_RR}
{\sc T.~Ikegami and T.~Sakurai}, {\em Contour integral eigensolver for
  {non-Hermitian} systems: a {Rayleigh--Ritz-type} approach}, Taiwanese J.
  Math., 14 (2010), pp.~825--837,
  \url{https://doi.org/10.11650/twjm/1500405869}.

\bibitem{SS_theory2010}
{\sc T.~Ikegami, T.~Sakurai, and U.~Nagashima}, {\em A filter diagonalization
  for generalized eigenvalue problems based on the {Sakurai--Sugiura}
  projection method}, J. Comput. Appl. Math., 233 (2010), pp.~1927--1936,
  \url{https://doi.org/10.1016/j.cam.2009.09.029}.

\bibitem{SS_Arnoldi}
{\sc A.~Imakura, L.~Du, and T.~Sakurai}, {\em A block {Arnoldi-type} contour
  integral spectral projection method for solving generalized eigenvalue
  problems}, Appl. Math. Lett., 32 (2014), pp.~22--27,
  \url{https://doi.org/10.1016/j.aml.2014.02.007}.

\bibitem{SS_theory2016}
{\sc A.~Imakura, L.~Du, and T.~Sakurai}, {\em Error bounds of {Rayleigh--Ritz}
  type contour integral-based eigensolver for solving generalized eigenvalue
  problems}, Numer. Algorithms, 71 (2016), pp.~103--120,
  \url{https://doi.org/10.1007/s11075-015-9987-4}.

\bibitem{SS_review}
{\sc A.~Imakura, L.~Du, and T.~Sakurai}, {\em Relationships among contour
  integral-based methods for solving generalized eigenvalue problems}, Jpn. J.
  Ind. Appl. Math., 33 (2016), pp.~721--750,
  \url{https://doi.org/10.1007/s13160-016-0224-x}.

\bibitem{Jackson_damp}
{\sc L.~O. Jay, H.~Kim, Y.~Saad, and J.~R. Chelikowsky}, {\em Electronic
  structure calculations for plane-wave codes without diagonalization}, Comput.
  Phys. Commun., 118 (1999), pp.~21--30,
  \url{https://doi.org/10.1016/S0010-4655(98)00192-1}.

\bibitem{jia1997refined}
{\sc Z.~Jia}, {\em Refined iterative algorithms based on {Arnoldi}'s process
  for large unsymmetric eigenproblems}, Linear Algebra Appl., 259 (1997),
  pp.~1--23, \url{https://doi.org/10.1016/S0024-3795(96)00238-8}.

\bibitem{jia_block}
{\sc Z.~Jia}, {\em Generalized block {Lanczos} methods for large unsymmetric
  eigenproblems}, Numer. Math., 80 (1998), pp.~239--266,
  \url{https://doi.org/10.1007/s002110050367}.

\bibitem{jia2000refined}
{\sc Z.~Jia}, {\em A refined subspace iteration algorithm for large sparse
  eigenproblems}, Appl. Numer. Math., 32 (2000), pp.~35--52,
  \url{https://doi.org/10.1016/S0168-9274(99)00008-2}.

\bibitem{CJ_SS_RR}
{\sc Z.~Jia and T.~Liu}, {\em A {Chebyshev--Jackson} series based block
  {SS--RR} algorithm for computing partial eigenpairs of real symmetric
  matrices}, arXiv preprint arXiv:2508.20456,  (2025).

\bibitem{CJ_SS_RRR}
{\sc Z.~Jia and T.~Liu}, {\em A refined {CJ--SS--RR} method for computing
  partial eigenpairs of {Hermitian} matrices}, arXiv preprint arXiv:2605.12846,
   (2026).

\bibitem{jiastewart2001}
{\sc Z.~Jia and G.~W. Stewart}, {\em An analysis of the {Rayleigh--Ritz} method
  for approximating eigenspaces}, Math. Comput., 70 (2001), pp.~637--647,
  \url{https://doi.org/10.1090/S0025-5718-00-01208-4}.

\bibitem{CJ_FEAST_cross}
{\sc Z.~Jia and K.~Zhang}, {\em A {FEAST} {SVDsolver} based on
  {Chebyshev--Jackson} series for computing partial singular triplets of large
  matrices}, J. Sci. Comput., 97 (2023), 21,
  \url{https://doi.org/10.1007/s10915-023-02342-y}.

\bibitem{CJ_FEAST_augmented}
{\sc Z.~Jia and K.~Zhang}, {\em An augmented matrix-based {CJ-FEAST}
  {SVDsolver} for computing a partial singular value decomposition with the
  singular values in a given interval}, SIAM J. Matrix Anal. Appl., 45 (2024),
  pp.~24--58, \url{https://doi.org/10.1137/23M1547500}.

\bibitem{CJ_FEAST_GSVD}
{\sc Z.~Jia and K.~Zhang}, {\em A {CJ-FEAST GSVD}solver for computing a partial
  gsvd of a large matrix pair with the generalized singular values in a given
  interval}, Numer. Math., 157 (2025), pp.~897--949,
  \url{https://doi.org/10.1007/s00211-025-01466-7}.

\bibitem{Jin_FEM_EM}
{\sc J.-M. Jin}, {\em The Finite Element Method in Electromagnetics}, John
  Wiley \& Sons, Inc., Hoboken, NJ, 3rd~ed., 2014.

\bibitem{FEAST_non_Hermitian}
{\sc J.~Kestyn, E.~Polizzi, and P.~T. Peter~Tang}, {\em {FEAST} eigensolver for
  {non-Hermitian} problems}, SIAM J. Sci. Comput., 38 (2016), pp.~S772--S799,
  \url{https://doi.org/10.1137/15M1026572}.

\bibitem{Martinsson_Tropp_2020}
{\sc P.-G. Martinsson and J.~A. Tropp}, {\em Randomized numerical linear
  algebra: foundations and algorithms}, Acta Numerica, 29 (2020), pp.~403--572,
  \url{https://doi.org/10.1017/s0962492920000021}.

\bibitem{Parlett_Symmetric}
{\sc B.~N. Parlett}, {\em The Symmetric Eigenvalue Problem}, SIAM,
  Philadelphia, PA, 1998, \url{https://doi.org/10.1137/1.9781611971163}.

\bibitem{FEAST_SI}
{\sc P.~T. Peter~Tang and E.~Polizzi}, {\em {FEAST} as a subspace iteration
  eigensolver accelerated by approximate spectral projection}, SIAM J. Matrix
  Anal. Appl., 35 (2014), pp.~354--390,
  \url{https://doi.org/10.1137/13090866X}.

\bibitem{FEAST2009}
{\sc E.~Polizzi}, {\em Density-matrix-based algorithm for solving eigenvalue
  problems}, Phys. Rev. B, 79 (2009), 115112,
  \url{https://doi.org/10.1103/PhysRevB.79.115112}.

\bibitem{ApproxIntroduction}
{\sc T.~J. Rivlin}, {\em An Introduction to the Approximation of Functions},
  Dover, New York, 1981.

\bibitem{Saad_Eigen}
{\sc Y.~Saad}, {\em Numerical Methods for Large Eigenvalue Problems, Revised
  Edition}, SIAM, Philadelphia, PA, 2011,
  \url{https://doi.org/10.1137/1.9781611970739}.

\bibitem{SS_restart1}
{\sc T.~Sakurai, Y.~Futamura, and H.~Tadano}, {\em Efficient parameter
  estimation and implementation of a contour integral-based eigensolver}, J.
  Algorithms Comput. Technol., 7 (2013), pp.~249--269,
  \url{https://doi.org/10.1260/1748-3018.7.3.249}.

\bibitem{SS_Hankel}
{\sc T.~Sakurai and H.~Sugiura}, {\em A projection method for generalized
  eigenvalue problems using numerical integration}, J. Comput. Appl. Math., 159
  (2003), pp.~119--128, \url{https://doi.org/10.1016/S0377-0427(03)00565-X}.

\bibitem{SS_RR}
{\sc T.~Sakurai and H.~Tadano}, {\em {CIRR}: a {Rayleigh--Ritz} type method
  with contour integral for generalized eigenvalue problems}, Hokkaido Math.
  J., 36 (2007), pp.~745--757, \url{https://doi.org/10.14492/hokmj/1272848031}.

\bibitem{Stewart_Eigen}
{\sc G.~W. Stewart}, {\em Matrix Algorithms, Volume II: Eigensystems}, SIAM,
  Philadelphia, PA, 2001, \url{https://doi.org/10.1137/1.9780898718058}.

\bibitem{book:VanDerVorst2002}
{\sc H.~A. van~der Vorst}, {\em {Computational Methods for Large Eigenvalue
  Problems}}, Handbook of Numerical Analysis, Vol. VIII, edited by P. G.
  Ciarlet and J. L. Lions, Elsevier, 2002,
  \url{https://doi.org/10.1016/S1570-8659(02)08003-1}.

\bibitem{FEAST_oblique}
{\sc G.~Yin, R.~H. Chan, and M.-C. Yeung}, {\em A {FEAST} algorithm with
  oblique projection for generalized eigenvalue problems}, Numer. Linear
  Algebra Appl., 24 (2017), e2092, \url{https://doi.org/10.1002/nla.2092}.

\bibitem{tangent_angle}
{\sc P.~Zhu and A.~Knyazev}, {\em Angles between subspaces and their tangents},
  J. Numer. Math., 21 (2013), pp.~325--340,
  \url{https://doi.org/10.1515/jnum-2013-0013}.

\end{thebibliography}

\end{document}